\documentclass[11pt]{amsart}
\usepackage[margin=1.15in]{geometry}
\usepackage{amscd,amssymb, amsmath, wasysym}

\usepackage{graphicx}
\usepackage{amsfonts}
\usepackage{mathrsfs}    
\usepackage{amsmath}    
\usepackage{amsthm}     
\usepackage{amscd}      
\usepackage{amssymb}    
\usepackage{eucal}      
\usepackage{latexsym}   
\usepackage{graphicx}   
\usepackage{verbatim}   
\usepackage[all,cmtip,line]{xy}    
\usepackage[colorlinks=true,linkcolor=blue,urlcolor=black,bookmarksopen=true]{hyperref}
\usepackage{bookmark}

\makeatletter

\newcounter{thmcounter}

\numberwithin{equation}{section}
\numberwithin{thmcounter}{section}

\newtheorem{theorem}[thmcounter]{Theorem}
\newtheorem{proposition}[thmcounter]{Proposition}
\newtheorem{lemma}[thmcounter]{Lemma}
\newtheorem{corollary}[thmcounter]{Corollary}

\theoremstyle{definition}
\newtheorem{definition}[thmcounter]{Definition}

\newtheorem{example}[thmcounter]{Example}

\newtheorem{remark}[thmcounter]{Remark}

\newtheoremstyle{claim}{9pt}{3pt}{}{\parindent}{\bf}{.}{1em}{}

\theoremstyle{claim}

\newenvironment{namelist}[1]{%
\begin{list}{}
{
\settowidth{\labelwidth}{#1}%
\setlength{\labelsep}{0.3em}%
\setlength{\leftmargin}{\labelwidth}%
\addtolength{\leftmargin}{\labelsep}}}{%
\end{list}}

\newcommand{\nZ}{\mathbb{Z}}                     

\newcommand{\nP}{\mathbb{P}}                     

\newcommand{\sF}{\mathscr{F}}

\newcommand{\sO}{\mathscr{O}}                    

\newcommand{\sL}{\mathscr{L}}

\newcommand{\mf}[1]{\mathfrak{#1}}

\DeclareMathOperator{\bl}{Bl}                    

\DeclareMathOperator{\coker}{Coker}              

\DeclareMathOperator{\gr}{gr}                    

\DeclareMathOperator{\id}{id}                    

\DeclareMathOperator{\pr}{pr} 

\DeclareMathOperator{\qis}{qis}                  

\DeclareMathOperator{\Supp}{Supp}                
\DeclareMathOperator{\Sym}{Sym}                  
\DeclareMathOperator{\sHom}{\mathscr{H}om}       

\DeclareMathOperator{\Proj}{Proj}                

\newcommand*{\longhookrightarrow}{\ensuremath{\lhook\joinrel\relbar\joinrel\rightarrow}}

\newcounter{rkcounter}             
\begin{document}

\title[Local and global properties of secants of curves]{Some local and global properties of secant varieties of nonsingular projective curves}

\author{Lawrence Ein}
\address{Department of Mathematics, University Illinois at Chicago, 851 South Morgan St.,
Chicago, IL 60607, USA}
\email{ein@uic.edu}

\author{Wenbo Niu}
\address{Department of Mathematical Sciences, University of Arkansas, Fayetteville, AR 72701, USA}
\email{wenboniu@uark.edu}

\author{Jinhyung Park}
\address{Department of Mathematical Sciences, KAIST, 291 Daehak-ro, Yuseong-gu, Daejeon 34141, Republic of Korea}
\email{parkjh13@kaist.ac.kr}



\date{\today}

\begin{abstract} 
The main goal of this paper is to study some local and global properties of secant varieties of algebraic curves. These results complement our previous work \cite{ENP} by addressing issues given therein and providing solutions to problems raised subsequently. Specifically, we show a description of tangent cones of secant varieties of curves, and compute the cohomology groups of secant sheaves on symmetric products of curves, which answers a question posed in \cite{ENP} and leads to a recursive formula for Hilbert polynomials of secant varieties of curves. In the appendix, we present a cohomological approach to arithmetical Cohen--Macaulayness of secant varieties of curves, completing the proof in \cite{ENP}.
\end{abstract}

\maketitle


\section{Introduction}

\noindent Throughout the paper, we work over an algebraically closed field $\Bbbk$ of characteristic zero.
 Let $C$ be a nonsingular projective curve of genus $g\geq 0$, and let $L$ be a very ample line bundle determining an embedding 
 $$
 C\longhookrightarrow \nP^r:=\nP(H^0(L)), \text{ where }r:=r(L)=h^0(L)-1.
 $$
 Assume that $\deg L\geq 2g+2k+1$ for an integer $k\geq 0$, and consider the \emph{$k$-th secant variety} 
 $$
 \Sigma_k=\Sigma_k(C, L) \subseteq \nP^r
 $$
 of $C \subseteq \nP^r$ defined as the union of $(k+1)$-secant $k$-planes to $C$. Note that $\dim \Sigma_k = 2k+1$ and the singular locus of $\Sigma_k$ is $\Sigma_{k-1}$ unless $(C, L) = (\nP^1, \sO_{\nP^1}(2k+1))$ in which case $\Sigma_k = \nP^{2k+1}$. It has attracted considerable attention in last two decades on understanding the geometry of secant varieties (see e.g., \cite{CKP, CS, CR, ENP, RS, SV, Ullery}). In our previous work \cite{ENP}, we studied singularities and syzygies of secant varieties of curves by settling several conjectures in this area. In particular, we proved that $\Sigma_k$ has normal Du Bois singularities, $\Sigma_k \subseteq \nP^r$ is projectively normal, and $H^i(\Sigma_k, \sO_{\Sigma_k}(\ell))=0$ for $i > 0$ and $\ell > 0$. The necessary basic known facts on secant varieties of curves are reviewed in Section \ref{sec:prelim}. In the sequel, we look at further local and global properties of secant varieties of curves, which are complementary results to \cite{ENP}. 
 
\medskip

The first result that we establish in this paper is the description of tangent cones of secant varieties of curves. The classical Terracini's lemma (see e.g., \cite[Theorem 1.1]{CR}) determines the tangent spaces of secant varieties at general points. In our situation, if $x \in \Sigma_k \setminus \Sigma_{k-1}$, then there exists a unique effective divisor $\xi \in C_{k+1}$ of degree $k+1$ such that $x \in \langle \xi \rangle$. Then Terracini lemma says that the projective tangent space to $\Sigma_k$ at $x$ is $\langle 2\xi \rangle$. 
If $x \in \Sigma_m \setminus \Sigma_{m-1}$ in $\Sigma_k$, then there exists a unique effective divisor $\xi \in C_{m+1}$ such that $x \in \langle \xi \rangle$. We actually prove that the ``projective'' tangent cone to $\Sigma_k$ at $x$ is a cone over the secant variety $\Sigma_{k-m-1}(C, L(-2\xi))$ with the vertex $\langle 2\xi \rangle$ (see Theorem \ref{thm:01}, in which the ``projectivized'' tangent cone is described). Our result naturally generalizes the classical Terracini's lemma. As an application, we give a quick proof of the fact that $\Sigma_k$ is Cohen--Macaulay (see Corollary \ref{cor:02}). This readily implies that $\Sigma_k \subseteq \nP^r$ is arithmetically Cohen--Macaulay (see Corollary \ref{cor:Sigma=aCM}), as was asserted in \cite[Theorem 5.8]{ENP}.
Along the way, we also show that all cones over secant varieties of curves have normal Cohen--Macaulay Du Bois singularities (see Theorem \ref{thm:coneisDB}). This answers a question of Karen Smith (see Remark \ref{rem:KS}).

\medskip

Our second result is the computation of the cohomology groups of symmetric powers of secant sheaves on symmetric products of algebraic curves. Essentially, secant varieties of curves are the images of secant bundles over symmetric products of curves, which are the projectivizations of secant sheaves. It was asked in \cite[Problem 6.2]{ENP} that if one can explicitly determine the cohomology groups of symmetric powers of those secant sheaves. We are able to answer this question using the higher direct images of the structure sheaves of the secant bundles (see Theorem \ref{thm:cohomologysecantsheaf}). It is common to investigate the properties of secant varieties through symmetric products of curves. However, in this case, the approach is reversed: we derive properties of symmetric products of curves from secant varieties. In the process, we establish a recursive formula for the Hilbert polynomials of secant varieties of curves (see Theorem \ref{thm:hilbertpolynomial}). For the completeness, we also compute the cohomology groups of exterior powers of secant sheaves (see Theorem \ref{thm:cohwedgesecsheaf}).

    

\medskip

In the appendix of this paper (Appendix \ref{sec:appendix}), we present a cohomological approach to arithmetical Cohen--Macaulayness of secant varieties of curves along the line of \cite[Theorem 5.8]{ENP} which asserts that $\Sigma_k \subseteq \nP^r$ is arithmetically Cohen--Macaulay. What we did in the proof is to show the cohomology vanishing  
\begin{equation}\label{eq:ACM}
	H^i(\Sigma_k, \sO_{\Sigma_k}(-\ell)) = 0~~\text{ for $1 \leq i \leq 2k$ and $\ell \geq 0$},
\end{equation}
which is equivalent to the arithmetical Cohen--Macaulayness since we already established that $\Sigma_k \subseteq \nP^r$ is projectively normal and $H^i(\Sigma_k, \sO_{\Sigma_k}(\ell))=0$ for $i > 0$ and $\ell > 0$. To achieve (\ref{eq:ACM}), we applied  \cite[Theorem 10.42]{K} to verify (\ref{eq:ACM}) only for the case of $\ell = 0,1$. Indeed, \cite[Theorem 10.42]{K} states that if $X$ is a projective pure dimensional scheme with Du Bois singularities and $A$ is an ample line bundle on $X$, then $h^i(X, A^{-\ell}) = h^i(X, A^{-1})$ for $i< \dim X$ and $\ell \geq 1$. However, this statement  contains a crucial typo: the equality $h^i(X, A^{-\ell}) = h^i(X, A^{-1})$  should be replaced by the inequality $h^i(X, A^{-\ell}) \geq h^i(X, A^{-1})$. Because of this, our proof is incomplete since we only proved (\ref{eq:ACM}) for $\ell=0,1$. This was pointed out by Doyoung Choi, to whom we would like to thank. In the appendix, we give a complete proof without quoting \cite[Theorem 10.42]{K}. Since our description of the tangent cones of secant varieties of curves also proves that secant varieties of curves are arithmetically Cohen–Macaulay (see Corollary \ref{cor:Sigma=aCM}), we have given two essentially different proofs of this statement.

\medskip

\noindent \textbf{Acknowledgment}.
We would like to thank Doyoung Choi and Karen Smith for providing some motivation for this paper.

\section{Symmetric products and secant varieties of curves}\label{sec:prelim}

\noindent In this section, we set up the notations used throughout the paper, and  briefly review facts on symmetric products, secant sheaves, secant bundles, and secant varieties of algebraic curves. For details, we refer the reader to the work \cite{ENP}. 

\medskip

We henceforth fix a nonsingular projective curve $C$ of genus $g \geq 0$. For an integer $k\geq 1$, we write $C_k$ for the $k$-th symmetric product of $C$, and write $C^k$ for the $k$-th  Cartesian (or ordinary) product of $C$. Consider the canonical morphism
$$
\sigma_{k+1} \colon C_{k} \times C\longrightarrow C_{k+1}
$$
sending $(\xi,x)$ to $\xi+x$. It is a finite flat surjective morphism of degree $k+1$.

\begin{definition}[{\cite[Definitions 3.1 and 3.10]{ENP}}]\label{def:secsheaf}
	For an integer $k\geq 1$, let $\pr_2 \colon C_k \times C \rightarrow C$ be the projection to $C$.
	For a line bundle $L$ on $C$, the \emph{secant sheaf} on $C_{k+1}$ associated to $L$\footnote{This is also known as the \emph{tautological bundle} on $C_{k+1}$ associated to $L$.} is defined to be 
	$$
	E_{k+1,L}:=\sigma_{k+1,*}(\pr_2^*L)=\sigma_{k+1,*} (\sO_{C_k} \boxtimes L),
	$$ 
	and the \emph{secant bundle} (of $k$-planes) over $C_{k+1}$  is defined to be
	$$
	B^{k}(L):=\nP(E_{k+1,L})
	$$ 
	equipped with the natural projection $\pi_k \colon B^{k}(L) \rightarrow C_{k+1}$. If the line bundle $L$ is clear from the context, then we sometimes simply write $B^k$ instead of $B^k(L)$.
\end{definition}

Since $\sigma_{k+1}$ is a finite flat surjective morphism of degree $k+1$, it is elementary to see that the secant sheaf $E_{k+1, L}$ is locally free of rank $k+1$. We denote $\delta_{k+1}$ the Cartier divisor on $C_{k+1}$ such that
$$\sO_{C_{k+1}}(-\delta_{k+1})=\det E_{k+1,\sO_C}.$$

\begin{definition}[{\cite[Definition 3.2]{ENP}}]  
For a line bundle $L$ on $C$, we define a line bundle $T_{k+1}(L)$\footnote{In some papers (e.g., \cite{Agostini:SecnatConj, CKP}) following \cite{ENP}, the line bundle $T_{k+1}(L)$ is denoted by $T_{k+1,L}$ or $S_{k+1,L}$.} on $C_{k+1}$ to be the invariant descend of the $(k+1)$-fold box product   
$$
L^{\boxtimes (k+1)}:=\underbrace{L\boxtimes  \cdots \boxtimes L}_{\text{$k+1$ times}}
$$ 
from the Cartesian product $C^{k+1}$ to the symmetric product $C_{k+1}$ under the action of the symmetric group $\mathfrak{S}_{k+1}$. We also define two more line bundles:
$$
N_{k+1,L}:=\det E_{k+1,L}~~\text{ and }~~A_{k+1,L}:=T_{k+1,L}(-2\delta_{k+1}).
$$
\end{definition}

Several properties of line bundles $T_{k+1}(L)$, $ N_{k+1,L}$, and $A_{k+1,L}$ have been discussed and proved in \cite{ENP}. Among others, we recall that $N_{k+1, L} = T_{k+1}(L)(-\delta_{k+1})$ and $\omega_{C_{k+1}} = N_{k+1,\omega_C}$ (see \cite[Remark 3.3]{ENP}). The cohomology groups of these two line bundles are well studied and summarized in the following lemma. 

\begin{lemma}[{\cite[Lemma 2.4]{Agostini:SecnatConj}, \cite[Lemma 3.7]{ENP}}]\label{lm:03}
	For any $i \geq 0$, we have
	\begin{align*}
		&H^i(C_m, N_{m,L}) = \wedge^{m-i} H^0(C, L) \otimes S^i H^1(C, L) \text{ and } \\
		&H^i(C_m, T_{m}(L)) = S^{m-i} H^0(C, L) \otimes \wedge^i H^1(C, L). &
	\end{align*}
\end{lemma}

The space of global section $H^0(C_{k+1},E_{k+1, L})$ turns out to be the same as  $H^0(C, L)$, and the fiber $E_{k+1,L}\otimes \Bbbk(\xi)$ at a point $\xi \in C_{k+1}$ can be identified with $H^0(\xi, L|_{\xi})$. 

\medskip

From now on, we always assume that $L$ is a given line bundle on $C$ whose degree satisfies the condition 
$$
\deg L\geq 2g+2k+1
$$
unless otherwise stated.
Then $L$ is very ample, so it gives an embedding
$$
 C\longhookrightarrow \nP^r:=\nP(H^0(L)), \text{ where }r:=r(L)=h^0(L)-1.
$$
In fact, the line bundle $L$ separates $2k+2$ points (or we may say that $L$ is $(2k+1)$-very ample), namely, the restriction map 
$$
H^0(C, L)\longrightarrow H^0(\xi, L|_{\xi})
$$
is surjective for all $\xi \in C_{2k+2}$. This, in particular, implies that the tautological line bundle $\sO_{B^k(L)}(1)$ of $B^k(L) = \nP(E_{k+1,L})$ is globally generated, and therefore, it induces a morphism 
$$
\beta_k \colon B^k(L)\longrightarrow \nP^r.
$$
The \emph{$k$-th secant variety} $\Sigma_k=\Sigma_k(C,L)$ of $C$ embedded by $L$ is the image of the morphism $\beta_k$. Because of degree assumption on $L$, the morphism $\beta_k$ is actually a resolution of singularities of the secant variety $\Sigma_k$. Geometrically, the $k$-th secant variety $\Sigma_k$ is the variety swept out by the $(k+1)$-secant $k$-planes to $C$ in $\nP^r$. In what follows, $\Sigma_k$ always means the $k$-th secant variety $\Sigma_k(C, L)$ of the curve $C$ embedded by the given line bundle $L$ with $\deg L \geq 2g+2k+1$. Similarly, $\Sigma_m$ is the $m$-th secant variety $\Sigma_m(C, L)$ of $C$ embedded by $L$ for every $0 \leq m \leq k-1$. Sometimes, we need to consider secant varieties of $C$ embedded by another line bundle $\mathscr{L}$. In this case, to specify the embedding line bundle $\mathscr{L}$, we denote by $\Sigma_m(\mathscr{L})$ the $m$-th secant variety of $C$ embedded by $\mathscr{L}$.

\medskip

The $k$-th secant variety $\Sigma_k$ is stratified by the lower order secant varieties, i.e., 
$$C=\Sigma_0\subseteq \Sigma_k\subseteq \Sigma_2\subseteq\cdots\subseteq \Sigma_k.$$
From this stratification, we observe that for a closed point $x\in \Sigma_k$, there exits a unique integer $0 \leq m\leq k$ such that $x\in \Sigma_m \setminus \Sigma_{m-1}$. Note that each $(m+1)$-secant $m$-plane to $C$ in $\nP^r$ is spanned by a degree $m+1$ effective divisor on $C$. Thus there exists a unique degree $m+1$ effective divisor $\xi_{m+1, x}$ on $C$  such that $x$ is in the $(m+1)$-secant $m$-plane spanned by $\xi_{m+1, x}$.  We call $\xi_{m+1, x}$ the \emph{degree $m+1$ divisor on $C$ determined by $x$} (see \cite[Definition 3.12]{ENP}). Write
$$
F_x:=\beta_k^{-1}(x)
$$
for the fiber of $\beta_k$ over the point $x$. The geometry of the fiber $F_x$, especially, the conormal bundle $N_{F_x/B^k}$ inside $B^k=B^k(L)$, have been well studied in \cite[Section 3]{ENP}, and we refer the reader to there for details. Here we recall from \cite[Proposition 3.13]{ENP} that  
$$
F_x \cong C_{k-m}~~\text{ and }~~N_{F_x/B^k}^* \cong \sO_{F_x}^{\oplus 2m+1} \oplus E_{k-m, L(-2\xi_{m+1,x})}. 
$$

\medskip

For $0 \leq m\leq k$, there is a natural morphism 
$$
\alpha_{k,m} \colon B^m(L)\times C_{k-m}\longrightarrow B^k(L)
$$
whose construction is recalled here.  For any $\xi_{m+1} \in C_{m+1}$ and $\xi_{k-m} \in C_{k-m}$, let $\xi:=\xi_{m+1} + \xi_{k-m} \in C_{k+1}$. 
Note that the $(m+1)$-secant $m$-plane $\nP(H^0(L|_{\xi_{m+1}}))$ spanned by $\xi_{m+1}$ is naturally embedded in the $(k+1)$-secant $k$-plane $\nP(H^0(L|_{\xi}))$ spanned by $\xi$. 
The map $\alpha_{k,m}$ embeds $\nP(H^0(L|_{\xi_{m+1}}))\times \{\xi_{k-m}\}$ into $\nP(H^0(L|_{\xi}))$. The image of the morphism $\alpha_{k,m}$ is called the \emph{relative secant variety of $m$-planes} in $B^k(L)$ and denoted by $Z_m^k$. If the number $k$ is clear from the context, then we simply write $Z_m$ instead of $Z^k_m$. Notice that $Z_{k-1} = \beta_k^{-1}(\Sigma_{k-1})$ is a prime divisor on $B^k(L)$ such that $\sO_{B^k(L)}(Z_{k-1}) = \sO_{B^k(L)}(k+1) \otimes \pi_k^* A_{k,L}^{-1}$ (see \cite[Proposition 3.15]{ENP}). It can be shown that the restricted morphism $\beta_k|_{B^k(L)\setminus Z_{k-1}} \colon B^k(L) \setminus Z_{k-1} \to \Sigma_k \setminus \Sigma_{k-1}$ is an isomorphism. This in particular implies that the singular locus of $\Sigma_k$ is exactly $\Sigma_{k-1}$.

\medskip

The morphism $\alpha_{k,m}$ is compatible with the morphisms $\beta_k$ and $\beta_m$, so one has a commutative diagram
$$
\xymatrix{
	B^m(L)\times C_{k-m} \ar[d]_{\pr_1} \ar[r]^-{\alpha_{m,k}} & B^k(L)\ar[d]^{\beta_{k}}\\
	B^m(L) \ar[r]_-{\beta_{m}} &\nP(H^0(L)),}
$$
where $\pr_1$ is the projection. Recall also that there is a natural projection $\pi_k \colon B^k(L)\rightarrow C_{k+1}$ in the definition of $B^k(L)$. We have the following fact on the dualizing sheaves, 
$$\pi^*_kT_{k+1}(\omega_C)=\omega_{B^k(L)}(Z_{k-1}),$$
which  can be found in \cite[(5.13)]{ENP}. 

\medskip

We now recall some further basic properties of secant varieties proved in \cite{ENP}. 

\begin{theorem}[{\cite[Theorem 5.2 and Proposition 5.4]{ENP}}]\label{p:01}
One has the following:
	\begin{enumerate}
		\item $\Sigma_k$ has normal Du Bois singularities.
		\item $\Sigma_k \subseteq \nP^r$ is projectively normal.
		\item $H^i(\Sigma_k, \sO_{\Sigma_k}(\ell))=0$ for $i > 0$ and $\ell > 0$. 
	\end{enumerate}
\end{theorem}

There is a Du Bois-type condition given in \cite[Theorem 5.2]{ENP}. We reformulate it in terms of complexes. As usual, a coherent sheaf $\sF$ on a variety can be considered as a complex with $\sF$ as a degree zero term in the derived category.

\begin{proposition}\label{p:21} 
One has the following:
	\begin{enumerate}
		\item $R\beta_{k,*}\sO_{B^k(L)}(-Z_{k-1})\cong_{\qis} I_{\Sigma_{k-1}/\Sigma_k}$.
		\item $H^i(C_{k+1},S^\ell E_{k+1,L}\otimes T_{k+1}(\omega_C))=H^{2k+1-i}(\Sigma_{k}, I_{\Sigma_{k-1}/\Sigma_{k}}(-\ell))^*$,  for $i\geq 0$ and $\ell\geq 0$.
		\item There is an exact triangle
		\begin{equation}\label{eq:tri}
			\omega^\bullet_{\Sigma_{k-1}}\longrightarrow \omega^\bullet_{\Sigma_k}\longrightarrow R\beta_{k,*}\omega_{B^k(L)}(Z_{k-1})[2k+1]\stackrel{+1}{\longrightarrow}
		\end{equation}
            in the derived category of $\Sigma_k$, where $\omega^\bullet$ is the dualizing complex of a variety.
	\end{enumerate}
\end{proposition}

\begin{proof} 
For the first statement (1), the quasi-isomorphism $R\beta_{k,*}\sO_{B^k(L)}(-Z_{k-1})\cong_{\qis} I_{\Sigma_{k-1}/\Sigma_k}$ was proved in \cite[Theorem 5.2]{ENP}.  For the second statement (2), write $H$ for the tautological line bundle of $B^k(L)$. We compute directly that
	\begin{eqnarray*}
		H^{2k+1-i}(\Sigma_{k}, I_{\Sigma_{k-1}/\Sigma_{k}}(-\ell))^*&=&H^{2k+1-i}(B^k(L),\sO_{B^k(L)}(-Z_{k-1}-\ell H))^*\\
		&=&H^i(B^k(L),\omega_{B^k(L)}(Z_{k-1}+\ell H)) \\
		&=&H^i(C_{k+1},S^\ell E_{k+1,L}\otimes T_{k+1}(\omega_C)).
	\end{eqnarray*} 
For the last statement (3), apply the functor $R\sHom(\ \_\ , \omega^\bullet_{\Sigma_{k}})$ to the short exact sequence 
	$$0\longrightarrow I_{\Sigma_{k-1}/\Sigma_k}\longrightarrow \sO_{\Sigma_k}\longrightarrow \sO_{\Sigma_{k-1}}\longrightarrow 0$$
	to yield an exact triangle 
	$$\omega^\bullet_{\Sigma_{k-1}}\longrightarrow \omega^\bullet_{\Sigma_k}\longrightarrow R\sHom(I_{\Sigma_{k-1}/\Sigma_k},\omega^\bullet_{\Sigma_{k}})\stackrel{+1}{\longrightarrow}$$
	By Grothendieck duality we have 
	$$R\sHom(I_{\Sigma_{k-1}/\Sigma_k},\omega^\bullet_{\Sigma_{k}})=R\sHom(R\beta_{k,*}\sO_{B^k(L)}(-Z_{k-1}),\omega^\bullet_{\Sigma_{k}})\cong R\beta_{k,*}\omega_{B^k(L)}(Z_{k-1})[2k+1],$$
	where we use the fact that the dualizing complex $\omega^\bullet_{B^k(L)}$ of $B^k(L)$ is $\omega_{B^k(L)}[2k+1]$.
	Thus we obtain the desired exact triangle. 
\end{proof}

\section{Tangent cones of secant varieties}\label{sec:tangentcones}

\noindent In this section, we study the tangent cones of secant varieties of curves. Our approach is motivated by \cite[Section 3]{Ein}. Throughout the section, we fix a closed point $x$ in the $k$-th secant variety $\Sigma_k$. If $x\notin \Sigma_{k-1}$, then $\Sigma_{k}$ is nonsingular at $x$ and its tangent cone at $x$ is just the tangent space, which is not interesting here. So we assume that $x\in \Sigma_{k-1}$ so that $\Sigma_k$ is singular at $x$. Then there exists a unique integer $m$ with $0\leq m\leq k-1$ such that $x\in \Sigma_m \setminus \Sigma_{m-1}$. There exists a unique effective divisor $\xi = \xi_{m+1, x} \in C_{m+1}$ of degree $m+1$ such that the point $x$ is in the linear space $\langle \xi \rangle$ spanned by $\xi$ in $\nP^r$. 

\medskip

Recall that there is  a birational morphism 
$$
\beta_k \colon B^k \longrightarrow \Sigma_k\subseteq \nP^r=\nP(H^0(L)).
$$
Write 
$$
F:=\beta^{-1}_k(x)
$$
for the fiber of $\beta_k$ over the point $x$. Recall from \cite[Proposition 3.13]{ENP} that $F\cong C_{k-m}$ and
and $N^*_{F/B^k}$ is locally free with a canonical splitting into a direct sum of   the conormal bundle $N^*_{F/Z_m}$ of $F$ inside $Z_m$, which is the relative secant variety of $m$-planes in $B^k$, and  the restriction of the conormal bundle  $N^*_{Z_m/B^k}$ of $Z_m$ inside $B^k(L)$, namely,
\begin{equation}\label{eq:split}
	N^*_{F/B^k}=N^*_{F/Z_m}\oplus N^*_{Z_m/B^k(L)}|_F.
\end{equation}
These two direct summands in the splitting can be computed explicitly: 
$$N^*_{Z_m/B^k}|_F\cong E_{k-m,L(-2\xi)} ~~\text{ and }~~ N^*_{F/Z_m}=T_x^*\Sigma_m\otimes \sO_F,$$
where  $T_x^*\Sigma_m$ is the cotangent space of $\Sigma_m$ at $x$, which has dimension $2m+1$ as $\Sigma_m$ is nonsingular at $x$.

\medskip

Write $\mf{m}$ for the maximal ideal sheaf of $\sO_{\Sigma_k}$ defining the point $x$. The cotangent space of $\Sigma_k$ at $x$ is  $T^*_x\Sigma_k=\mf{m}/\mf{m}^2$. Recall that $\Sigma_k$ is singular at $x$. It was observed in \cite{ENP} that  
$$T^*_x\Sigma_k=T^*_x\nP^r=H^0(F,N^*_{F/B^k}).$$
Consider the associated graded ring of the local ring $\sO_{\Sigma_{k},x}$ which is defined to be $$G:=\gr_{\mf{m}}\sO_{\Sigma_k}=\oplus_{\ell\geq 0} \mf{m}^\ell/\mf{m}^{\ell+1}.$$
Then the ``projectivized'' tangent cone to $\Sigma_k$ at $x$ is defined to be a projective scheme $$\nP(G):=\Proj(G).$$ It is naturally a projective subscheme of the projectivized tangent space $\nP^{r-1}:=\nP(T^*_x\nP^r)$ of $\nP^r$ at $x$, which is the same as the projectivized tangent space of $\Sigma_k$ at $x$, determined by the surjective graded morphism $\Sym^*T^*_x\nP^r \rightarrow G$.
We will frequently just call $\nP(G)$ as the tangent cone. Here we mention a geometric consequence of the splitting in (\ref{eq:split}). It gives two natural morphisms 
$$\nP(E_{k-m,L(-2\xi)})\longrightarrow \nP(T^*_x\nP^r) ~~\text{ and }~~ \nP(T_x^*\Sigma_m\otimes \sO_F)\longrightarrow \nP(T^*_x\nP^r).$$
The image of the first map is the secant variety $\Sigma_{k-m-1}(L(-2\xi))$ and the image of the second one is simply the linear space $\nP(T_x^*\Sigma_m)$.

\begin{lemma}\label{lem:gammaisomorphism}
For any integer $\ell\geq 0$, the natural map
	$$\gamma_\ell \colon \mf{m}^\ell/\mf{m}^{\ell+1}\longrightarrow H^0(F,S^\ell N^*_{F/B^k})$$
	is an isomorphism. 
\end{lemma}
\begin{proof}
The map $\gamma_{\ell}$ fits into the following commutative diagram 
	$$\xymatrix{
		S^\ell T^*_x\nP^r \ar[dr]_-{\alpha_\ell} \ar[r] & \mf{m}^\ell/\mf{m}^{\ell+1}\ar[d]^{\gamma_\ell}\\
		& H^0(F,S^\ell N^*_{F/B^k}).}$$
	As shown in the proof of \cite[Theorem 5.2]{ENP}, the morphism $\alpha_\ell$ is surjective. Thus we see immediately that the map $\gamma_\ell$ is also surjective. In the rest of the proof, we will show that $\gamma_\ell$ is injective. 
	
    As $\Sigma_k$ is normal, the natural map
    $\beta^\# \colon \sO_{\Sigma_k}\rightarrow \beta_{*}\sO_{B^k}$ is an isomorphism. This was explicitly proven in the proof of \cite[Theorem 5.2]{ENP} using the formal function. Consider two inverse systems $\{\sO_{\Sigma_k}/\mf{m}^\ell\}_{\ell\geq 0}$ and $H^0(\sO_{B^k}/I^\ell_F)\}_{\ell\geq 0}$. For each $\ell \geq 0$, there is a natural morphism 
	$$\beta^\ell_k \colon \sO_{\Sigma_k}/\mf{m}^\ell\longrightarrow H^0(\sO_{B^k}/I^\ell_F)$$
	induced by the morphism $\beta_k$. Write the kernel of $\beta^\ell_k$ as $K_\ell$ and we obtain an induced inverse system $\{K_\ell\}_{\ell\geq 0}$. Consider the commutative diagram 
	$$\xymatrix{
		0\ar[r]&	\mf{m}^{\ell}/\mf{m}^{\ell+1} \ar[d]^{\gamma_{\ell}} \ar[r] & \sO_{\Sigma_k}/\mf{m}^{\ell+1} \ar[r]\ar[d]^{\beta^{\ell+1}_k}  & \sO_{\Sigma_k}/\mf{m}^{\ell}\ar[d]^{\beta^\ell_k} \ar[r] & 0\\
		0\ar[r]&	H^0(I_F^{\ell}/I_F^{\ell+1}) \ar[r]& H^0(\sO_{B^k}/I^{\ell+1}_F)\ar[r]  & H^0(\sO_{B^k}/I^{\ell}_F)\ar[r] &\cdots
	}$$
	Since $\gamma_\ell$ is surjective, by Snake lemma, the induced map $K_{\ell+1}\rightarrow K_\ell$ is surjective. 
	
	As $\beta^\#$ is isomorphic, the formal function theorem then gives the isomorphism between the inverse limits
	$$\varprojlim (\sO_{\Sigma_k}/\mf{m}^\ell)\longrightarrow \varprojlim H^0(\sO_{B^k}/I^\ell_F).$$
	Therefore the inverse limits of the system $\{K_\ell\}_{\ell\geq 0}$ is zero, i.e., $\varprojlim K_\ell=0$. But  $K_{\ell+1}\rightarrow K_\ell $ is surjective for all $\ell\geq 0$. Hence we conclude that $K_\ell=0$ for all $\ell\geq 0$. This means the map $\beta^\ell_k$ is injective and by the Snake lemma, we see the map $\gamma_\ell$ is injective.
\end{proof}

The following is the main result of this section.

\begin{theorem}\label{thm:01} 
The projectivized tangent cone $\nP(G)$ in the space $\nP^{r-1}:=\nP(T^*_x\nP^r)$ is a cone over the secant variety $\Sigma_{k-m-1}(L(-2\xi))$ with the vertex $\nP(T_x^*\Sigma_m)$.
\end{theorem}

\begin{proof} 
Consider the blowup $\nP^r$ along $x$ and the blowup $B^k(L)$ along $F$. Consider the strict transforms of $\Sigma_k$ and $Z_m$ to yield the following commutative diagram 
	$$\xymatrix{
		\bl_FZ_m \ar@{^{(}->}[r]\ar[d] &	\bl_F(B^k)  \ar[d] \ar[r] & \bl_x\Sigma_k\ar[d]\ar@{^{(}->}[r] &  \bl_x\nP^r \ar[d] \\
		Z_m  \ar@{^{(}->}[r]&	B^k(L) \ar[r]^{\beta_k} &  \Sigma_k \ar@{^{(}->}[r] &  \nP^r. 
	}$$
	If we look at the exceptional divisors in the above diagram, we obtain the following induced commutative diagram 
	$$\xymatrix{
		\nP(N^*_{F/Z_m}) \ar@{^{(}->}[r]\ar[d] &	\nP(N^*_{F/B^k}) \ar[d] \ar[r] & \nP(G)\ar[d]\ar@{^{(}->}[r] &  \nP^{r-1}=\nP(T^*_{x}\nP^r) \ar[d] \\
		F  \ar@{^{(}->}[r]&	F \ar[r]^{\beta_k} &  x \ar@{^{(}->}[r] &  x .
	}$$
	Morphisms in the top row of the diagram from the schemes  to the space $\nP^{r-1}$ are all given  by the tautological bundle of the Proj construction. In particular, the composition map 
    $$\phi \colon \nP(N^*_{F/B^k})\longrightarrow \nP^{r-1}$$ factoring through $\nP(G)$ is determined by the global sections $$H^0(\sO_{\nP(N^*_{F/B^k})}(1))=H^0(N^*_{F/B^k})=T_x^*\Sigma_m\oplus H^0(E_{k-m,L(-2\xi)}).$$
     Recall the splitting in (\ref{eq:split}) that 
	$$N^*_{F/B^k}=N^*_{F/Z_m}\oplus N^*_{Z_m/B^k(L)}|_F.$$
	We see that the image $\phi(\nP(N^*_{F/Z_m}))$ is the secant variety $\Sigma_{k-m-1}(L(-2\xi))$ of $C$ determined by the line bundle $L(-2\xi)$ and the image $\phi(N^*_{F/Z_m})$ is the linear space $\nP(T_x^*\Sigma_m)$. Thus the image of $\phi$ is a cone over $\Sigma_{k-m-1}(L(-2\xi))$ with the vertex $\nP(T_x^*\Sigma_m)$, which at least set-theoretically equals $\nP(G)$.
	
	Next, we claim that $\nP(G)$ is the scheme-theoretical image of $\phi$. It is enough to show that the induced map 
	$$\phi^\# \colon \sO_{\nP(G)}\longrightarrow \phi_*\sO_{\nP(N^*_{F/B^k})}$$
	is injective. It is then enough to show that for $\ell \gg 0$, the induced map on the global sections
	$$H^0(\sO_{\nP(G)}(\ell))\longrightarrow H^0(\sO_{\nP(N^*_{F/B^k})}(\ell))$$
	is injective. But this map is exactly the map 
	$$\gamma^\ell \colon \mf{m}^\ell/\mf{m}^{\ell+1}\longrightarrow H^0(F,S^\ell N^*_{F/B^k})$$
        in Lemma \ref{lem:gammaisomorphism}, which is an isomorphism.
\end{proof}

\begin{remark}
We also see that the ``projective'' tangent cone to $\Sigma_k$ at $x$ is a cone over $\Sigma_{k-m-1}(C, L(-2\xi))$ with the vertex $\langle 2\xi \rangle$. Here Terracini lemma says that $\langle 2\xi \rangle$ is the projective tangent space to $\Sigma_m$ at $x$. We may think that $\Sigma_{k-m-1}(C, L(-2\xi))$ is the $(k-m-1)$-th secant variety of the embedded projective curve $C \subseteq \nP H^0(L(-2\xi))$ obtained by the inner projection of $C \subseteq \nP^r=\nP H^0(L)$ centered at $\langle 2\xi \rangle$. 
A similar result was also shown by Ciliberto--Russo \cite[Theorem 3.1]{CR}. They actually proved that the projective tangent cones of secant varieties are in general contained in cones over lower secant varieties. 
\end{remark}

\begin{remark}
The birational morphism $\phi \colon \nP (N_{F/B^k}^*) \to \nP (G)$ in the proof of Theorem \ref{thm:01} is a resolution of singularities. On the other hand, there is another birational morphism $P_G \to \nP(G)$ given by $|\sO_{P_G}(1)|$, where
$$
P_G:=\nP (\sO_{\Sigma_{k-m-1}(L(-2\xi))}^{\oplus 2m} \oplus \sO_{\Sigma_{k-m-1}(L(-2\xi))}(1)).
$$
By \cite[Proposition 2.9]{CG}, $P_G$ is of Fano type if and only if $\Sigma_{k-m-1}(L(-2\xi))$ is of Fano type. Recall from \cite[Theorem 1.1]{ENP} that $\Sigma_{k-m-1}(L(-2\xi))$ is of Fano type if and only if $C=\nP^1$. In this case, $\Sigma_{k-m-1}(L(-2\xi))$ is a Fano variety of Picard number $1$, and so is the tangent cone $\nP(G)$. One can also prove the converse: If $\nP(G)$ is a Fano variety, then $C=\nP^1$.      
\end{remark}

\begin{corollary}\label{cor:01} 
The associated graded ring $\gr_{\mf{m}}\sO_{\Sigma_k}$ is the coordinator ring of the projectivized tangent cone $\nP(G)$ in the space $\nP^{r-1}:=\nP(T^*_x\nP^r)$. 
\end{corollary}

\begin{proof}  
Note that $\deg L(-2\xi) \geq 2g+2(k-m-1)+1$. 
Since $\Sigma_{k-m-1}(L(-2\xi)) \subseteq \nP H^0(L(-2\xi))$ is projectively normal by Theorem \ref{p:01} and since $\nP(G)$ is a cone over the secant variety $\Sigma_{k-m-1}(L(-2\xi))$ with the vertex $\nP(T_x^*\Sigma_m)$,  we see that $\nP(G)$  is projectively normal in the space $\nP^{r-1}$ . Thus for each $\ell\geq 0$, we have a commutative diagram 
	$$\xymatrix{
		H^0(\nP^{r-1},\sO_{\nP^{r-1}}(\ell))=S^l(T^*_x\nP^r) \ar[dr]_-{\lambda_\ell} \ar[r] & \mf{m}^\ell/\mf{m}^{\ell+1}\ar[d]^{\mu_\ell}\\
		& H^0(\nP(G),\sO_{\nP(G)}(\ell))}.$$
	in which $\lambda_\ell$ is surjective and $\mu_{\ell}$ is injective. 	This implies that $\mu_{\ell}$ is also surjective and therefore
	$$H^0(\nP(G),\sO_{\nP(G)}(\ell))=\mf{m}^\ell/\mf{m}^{\ell+1}.$$
	This proves that the coordinate ring of $\nP(G)$ is the associated algebra $\gr_{\mf{m}}\sO_{\Sigma_k}$.
\end{proof}

As an interesting application of the description of tangent cones, we give a proof for the result that  secant varieties are arithmetical Cohen--Macaulay. This result was asserted in \cite{ENP} by a cohomological approach (we will also discuss it in details in Appendix \ref{sec:appendix}). First, we show that this can be reduced to proving the secant varieties are ``locally'' Cohen--Macaulay.

\begin{lemma}\label{lem:CM=>aCM}
Fix an integer $k \geq 0$. For each $0 \leq m \leq k$, assume that $\Sigma_m(\sL)$ is Cohen--Macaulay for any line bundle $\sL$ on $C$ with $\deg \sL \geq 2g+2m+1$. Then $\Sigma_k = \Sigma_k(L) \subseteq \nP^r = \nP H^0(C, L)$ is arithmetically Cohen--Macaulay for any line bundle $L$ on $C$ with $\deg L \geq 2g+2k+1$.
\end{lemma}

\begin{proof}
When $k=0$, the assertion is just the well-known theorem of Castelnuovo that $C \subseteq \nP H^0(L)$ is arithmetically Cohen--Macaulay. By induction on $k$, for each $0 \leq m \leq k-1$, we may assume that $\Sigma_m(\sL) \subseteq \nP H^0(\sL)$ is arithmetically Cohen--Macaulay when $\deg \sL \geq 2g+2m+1$. By Theorem \ref{p:01}, it is enough to show the vanishing 
\begin{equation}\label{eq:ACMinLemCM=>aCM}
H^i(\Sigma_{k},\sO_{\Sigma_{k}}(\ell))=0 \text{ for $1 \leq i \leq 2k$ and $\ell \leq 0$}.
\end{equation}
Since $\Sigma_k$ is Cohen--Macaulay, we have (\ref{eq:ACMinLemCM=>aCM}) for $\ell \ll 0$. 
By \cite[Theorem 9.12]{Kollar.Shafarevich}, we get (\ref{eq:ACMinLemCM=>aCM}) for $\ell \leq 0$. The only left case is when $\ell=0$. We only need to show that 
$$
H^i(\Sigma_{k},\sO_{\Sigma_{k}})=0~~\text{ for $1\leq i\leq 2k$}.
$$
By induction on $k$, we may assume that $H^i(\Sigma_{k-1}, \sO_{\Sigma_{k-1}})=0$ for $1 \leq i \leq 2k-2$. Chasing through the short exact sequence 
$$
0\longrightarrow I_{\Sigma_{k-1}/\Sigma_k}\longrightarrow \sO_{\Sigma_{k}}\longrightarrow \sO_{\Sigma_k}\longrightarrow 0,
$$
it is then sufficient to show that
	\begin{enumerate}
		\item $H^i(\Sigma_k,I_{\Sigma_{k-1}/\Sigma_k})=0, \text{ for }1\leq i\leq 2k-1, \text{ and }$
		\item the connection map 
		$$\tau \colon H^{2k-1}(\Sigma_k, \sO_{\Sigma_{k-1}})\longrightarrow H^{2k}(\Sigma_k,I_{\Sigma_{k-1}/\Sigma_k})$$
		is an isomorphism. 
	\end{enumerate}
By Proposition \ref{p:21}, we have 
$$
H^i(\Sigma_k,I_{\Sigma_{k-1}/\Sigma_k})^*=H^{2k+1-i}(C_{k+1},T_{k+1}(\omega_C)) \text{ for all }i\geq 0.
$$
Then (1) follows from Lemma \ref{lm:03} which gives 
$$
H^{2k+1-i}(C_{k+1},T_{k+1}(\omega_C))=0, \text{ for }1\leq i\leq 2k-1.
$$
For (2), applying the exact triangle (\ref{eq:tri}) for $\Sigma_{k-1}$ and taking the $(2k-1)$-th cohomology we compute the dualizing sheaf of $\Sigma_{k-1}$ as
	$$\omega_{\Sigma_{k-1}}=\beta_{k-1,*}\omega_{B^{k-1}(L)}(Z_{k-2}).$$ So we compute by Serre duality that
	$$H^{2k-1}(\Sigma_k, \sO_{\Sigma_{k-1}})^*=H^0(\Sigma_{k-1},\omega_{\Sigma_{k-1}})=H^0(B^{k-1}(L),\omega_{B^{k-1}(L)}(Z_{k-2}))=H^0(C_{k},T_{k}(\omega_C))$$
	which equals $H^1(C_{k+1},T_{k+1}(\omega_C))=H^{2k}(\Sigma_k,I_{\Sigma_{k-1}/\Sigma_k})^*$.	Thus $\tau$ is an isomorphism. 
\end{proof}

We show that the secant varieties are indeed locally Cohen--Macaulay. 

\begin{corollary}\label{cor:02} The secant variety  $\Sigma_k$ is Cohen--Macaulay.
\end{corollary}

\begin{proof} 
We proceed by induction on $k$. When $k=0$, the secant variety $\Sigma_0$ is just the curve $C$ itself, so $\Sigma_0$ is Cohen--Macaulay. By induction, we may assume that for each $0 \leq m \leq k-1$, the secant variety $\Sigma_m(\sL)$ is Cohen--Macaulay for any line bundle $\sL$ on $C$ with $\deg \sL \geq 2g+2m+1$. By Lemma \ref{lem:CM=>aCM},  $\Sigma_m(\sL)$ is arithmetically Cohen--Macaulay in $\nP H^0(C, \sL)$. Take a point $x\in \Sigma_k$. If $x\in \Sigma_k-\Sigma_{k-1}$, then $\Sigma_k$ is nonsingular at $x$ and thus it is Cohen--Macaulay at $x$. So we consider the case when $x\in \Sigma_{k-1}$.  By Theorem \ref{thm:01}, the tangent cone $\nP(G)$ of $\Sigma_k$ at $x$ is the cone over the secant variety $\Sigma_{k-m-1}(L(-2\xi))$ in the space $\nP(T^*_x\nP^r)$. But by induction, $\Sigma_{k-m-1}(L(-2\xi))$ is arithmetical Cohen--Macaulay in $\nP H^0(C, L(-2\xi))$, and thus, $\nP(G)$ is arithmetical Cohen--Macaulay in $\nP^{r-1}$. Corollary \ref{cor:01} says that the associated graded ring $\gr_{\mf{m}}(\sO_{\Sigma_{k}})$ is the coordinate ring of $\nP(G)$. Thus $\gr_{\mf{m}}(\sO_{\Sigma_{k}})$ is Cohen--Macaulay, and therefore, $\Sigma_k$ is Cohen--Macaulay at the point $x$ (see Remark \ref{rem:grCM} below).
\end{proof}

\begin{remark}\label{rem:grCM}
In the proof of Corollary \ref{cor:02}, we use a well-known commutative algebra fact: For a local noetherian ring $(R,\mf{m},k)$ such that the residue class field $k$ has characteristic zero, if the associated graded ring $\gr_{\mf{m}}(R)=\oplus_{l\geq 0} \mf{m}^l/\mf{m}^{l+1}$ is Cohen--Macaulay, then so is the local ring $R$. As we could not find a suitable reference, we give a sketch of the proof. Since $\gr_{\mf{m}}R$ is generated by $\mf{m}/\mf{m}^2$, we can take general linear forms $s^*_1,\cdots, s^*_n$, where $n=\dim R$ as a regular sequence of $\gr_{\mf{m}}R$. Then they determine a regular sequence $s_1,\cdots, s_n\in \mf{m}$, which proves the result. 
\end{remark}

\begin{corollary}\label{cor:Sigma=aCM}
The secant variety $\Sigma_k \subseteq \nP^r$ is arithmetically Cohen--Macaulay, i.e., one has the following:
\begin{enumerate}
\item $\Sigma_k \subseteq \nP^r$ is projectively normal.
\item $H^i(\Sigma_k, \sO_{\Sigma_k}(\ell))=0$ for $1 \leq i \leq 2k$ and $\ell \in \nZ$. 
\end{enumerate}
\end{corollary}

\begin{proof}
Immediate from Lemma \ref{lem:CM=>aCM} and Corollary \ref{cor:02}.
\end{proof}

Finally, we consider a cone $\overline{P} \subseteq \nP^{r+m}$ over the secant variety $\Sigma_k \subseteq \nP^r$ with the vertex $\Lambda \subseteq \nP^{r+m}$, which is a linear subspace of dimension $m-1$. Let
$$
P:=\nP (\sO_{\Sigma_k}(1) \oplus \sO_{\Sigma_k}^{\oplus m})
$$
with canonical projection $\tau \colon P \to \Sigma_k$. Then $\lvert \sO_P(1) \rvert$ induces a birational morphism 
$$
\psi \colon P \longrightarrow \overline{P} \subseteq \nP^{r+m}
$$
contracting an effective divisor $E:=\nP(\sO_{\Sigma_k}^{\oplus m}) \subseteq P$ to to the vertex $\Lambda \subseteq \nP^{r+m}$ of $\overline{P}$. Note that $\psi$ is the blow-up of $\overline{P}$ along $\Lambda$ with exceptional divisor $E$.

\begin{theorem}\label{thm:coneisDB}
The cone $\overline{P}$ has normal Cohen--Macaulay Du Bois singularities, and $\overline{P} \subseteq \nP^{r+m}$ is arithmetically Cohen--Macaulay. Consequently, the affine cone over $\Sigma_k$ and all the tangent cones to $\Sigma_k$ have normal Cohen--Macaulay Du Bois singularities.
\end{theorem}

\begin{proof}
Since $\Sigma_k \subseteq \nP^r$ is arithmetically Cohen--Macaulay, it follows that $\overline{P} \subseteq \nP^{r+m}$ is also arithmetically Cohen--Macaulay. In particular, $\overline{P}$ has normal Cohen--Macaulay singularities. It only remains to show that $\overline{P}$ has Du Bois singularities. Notice that $P, E$, and $\Lambda$ have Du Bois singularities. By Koll\'{a}r--Kov\'{a}cs criterion \cite[Corollary 6.28]{K}, it is enough to check that $R^i \psi_* \sO_P(-E)=0$ for $i > 0$, which is equivalent to that
$H^i(P, \sO_P(\ell H-E))=0$ for $i>0$ and $\ell \gg 0$, where $H$ is a tautological divisor on $P$. Since $\sO_P(-E) = \sO_P(-1) \otimes \tau^* \sO_{\Sigma_k}(1)$ and $H^i(\Sigma_k, \sO_{\Sigma_k}(m))=0$ for $i>0$ and $m >0$, we have
\[
H^i(P, \sO_P(\ell H-E)) = H^i(\Sigma_k, S^{\ell-1} (\sO_{\Sigma_k}(1) \oplus \sO_{\Sigma_k}^{\oplus m}) \otimes \sO_{\Sigma_k}(1)) = 0~~\text{ for $i>0$ and $\ell > 0$}. \qedhere
\]
\end{proof}

\begin{remark}\label{rem:KS}
Karen Smith personally asked the first author of the paper whether the affine cones over the secant varieties of curves have Du Bois singularities. The second statement of Theorem \ref{thm:coneisDB} provides an answer to this question.
\end{remark}

\section{Cohomology of secant sheaves}\label{sec:cohomology}

\noindent In this section, we compute the cohomology groups of symmetric powers of secant sheaves. Recall that the secant sheaf $E_{k+1,L}$ on $C_{k+1}$ associated to $L$ is a locally free sheaf of rank $k+1$. Let us continue to assume that $\deg L\geq 2g+2k+1$. Recall also that there is a birational morhpism $$\beta_k \colon B^k(L) \longrightarrow \Sigma_k\subseteq \nP^r,$$
where $B^k(L)=\nP(E_{k+1,L})$ is the secant bundle over $C_{k+1}$ equipped with a projection $\pi_k \colon B^k(L) \to C_{k+1}$. One has a commutative diagram 
\begin{equation}\label{eq:16}
		\begin{split}
			\xymatrixcolsep{0.7in}
			\xymatrix{
				B^{k-1}(L)\times C \ar[r]^-{\alpha_{k,k-1}} \ar[d]_-{\pr_1} & B^k(L) \ar[d]^-{\beta_k} \\
				B^{k-1}(L) \ar[r]_-{\beta_{k-1}} & \Sigma_{k-1},
			}
		\end{split}
\end{equation}

We start with the following simple observation. 

\begin{lemma}\label{lm:01} 
Consider the morphism
	$$\alpha_{k,k-1} \colon B^{k-1}(L)\times C\longrightarrow B^k(L)$$
	whose image is $Z_{k-1}$. Then $\alpha_{k,k-1}$ is finite  and 
	$$\alpha^*_{k,k-1}\omega_{Z_{k-1}}=\alpha^*_{k,k-1}\omega_{B^k(L)}(Z_{k-1})=\omega_{B^{k-1}(L)}(Z_{k-2})\boxtimes \omega_C.$$
\end{lemma}
\begin{proof} We have a commutative diagram 
\begin{equation}\label{eq:16}
		\begin{split}
			\xymatrixcolsep{0.7in}
			\xymatrix{
				B^{k-1}(L)\times C \ar[r]^-{\alpha_{k,k-1}} \ar[d]_-{\pi_{k-1}\times \id} & B^k(L) \ar[d]^-{\pi_k} \\
				C_{k}\times C \ar[r]_-{\sigma_{k-1,1}} & C_{k+1},
			}
		\end{split}
\end{equation}
By adjunction formula, $\omega_{Z_{k-1}}=\omega_{B^k(L)}(Z_{k-1})|_{Z_{k-1}}$. But we know that $\omega_{B^k(L)}(Z_{k-1})=\pi^*_kT_{k+1}(\omega_C)$. So it is enough to show the equality 
$$(\pi_{k-1}\times \id)^*(\sigma_{k-1,1}^*(T_{k+1}(\omega_C)))=\omega_{B^{k-1}(L)}(Z_{k-2})\boxtimes \omega_C.$$
This is true because 
\[
\sigma_{k-1,1}^*(T_{k+1}(\omega_C))=T_{k}(\omega_C)\boxtimes \omega_C, \text{ and } \pi^*_{k-1}T_{k}(\omega_C)=\omega_{B^{k-1}(L)}(Z_{k-2}).\qedhere
\]	
\end{proof}

The following is the main result of this section, which gives an answer to \cite[Problem 6.2]{ENP}. 

\begin{theorem}\label{thm:cohomologysecantsheaf}
For each integer $0 \leq i \leq k$, let  $\Sigma_i:=\Sigma_{i}(L)$ be the $i$-th secant variety of $C$ in the space $\nP^r = \nP (H^0(L))$, where $L$ is a line bundle on $C$ with $\deg L \geq 2g+2k+1$.
One has the following:
	\begin{enumerate}
		\item $R^i \beta_{k,*} \sO_{B^k(L)} = \begin{cases} \wedge^i H^1(C, \sO_C) \otimes \sO_{\Sigma_{k-i}} & \text{for $0 \leq i \leq k$} \\ 0 & \text{for $i \geq k+1$}. \end{cases}$
		\item For integers $0 \leq i \leq k+1$ and $\ell > 0$, we have
		$$
		H^i(C_{k+1}, S^{\ell} E_{k+1, L}) = H^i(B^k(L), \sO_{B^k}(\ell))=\wedge^i H^1(C, \sO_C) \otimes H^0(\Sigma_{k-i}, \sO_{\Sigma_{k-i}}(\ell)).
		$$
		In particular,  $H^{k+1}(C_{k+1}, S^{\ell} E_{k+1, L}) =0$ for $\ell > 0$. 
	\end{enumerate}
\end{theorem}

\begin{proof}
	First, we prove that $(1)$ implies $(2)$. It is clear that 
    $$
    H^i(C_{k+1}, S^{\ell} E_{k+1, L}) = H^i(B^k(L), \sO_{B^k}(\ell))~\text{ for $i \geq 0$ and $\ell>0$}.
    $$
    Note that $H^j(\Sigma_{k-i}, \sO_{\Sigma_{k-i}}(\ell))=0$ for $j>0$ and $\ell > 0$. By (1), we have $H^i(B^k(L), \sO_{B^k}(\ell)) = \wedge^i H^1(C, \sO_C) \otimes H^0(\Sigma_{k-i}, \sO_{\Sigma_{k-i}}(\ell))$. 
	
	\medskip
	
	Now, we prove $(1)$. Let $x \in \Sigma_m \setminus \Sigma_{m-1}$ in $\Sigma_k$. There is $\xi_x \in C_{m+1}$ such that $x \in \langle \xi_x \rangle$. Recall that $F_x:=\beta_k^{-1}(x) \cong C_{k-m}$. Then $\Supp (R^i \beta_{k,*} \sO_{B^k(L)})\subseteq \Sigma_{k-i}$. In particular,  $R^i \beta_{k,*} \sO_{B^k(L)} = 0$ for $i \geq k+1$. As $\Sigma_k$ is normal, we get $\beta_{k,*} \sO_{B^k(L)} = \sO_{\Sigma_k}$. We henceforth focus on the cases $1 \leq i \leq k$. We need to assume that $0 \leq m \leq k-i$. We proceed by induction on $k$. The case of $k=1$ is trivial. Assume that $k \geq 2$. Consider the commutative diagram (\ref{eq:16}) in which $\alpha_{k,k-1}$ is a finite by Lemma \ref{lm:01}. As
	$$
	R^i \beta_{k,*} \alpha_{k,k-1,*} \sO_{B^{k-1}(L) \times C} = R^i (\beta_k \circ \alpha_{k,k-1})_* \sO_{B^{k-1}(L) \times C}  = R^i (\beta_{k-1} \circ \pr_1)_* \sO_{B^{k-1}(L) \times C},
	$$
	the map $\sO_{B^k(L)} \rightarrow \alpha_{k,k-1,*} \sO_{B^{k-1}(L) \times C}$ gives rise to a map
	$$
	R^i \beta_{k,*} \sO_{B^k(L)} \longrightarrow R^i (\beta_{k-1} \circ \pr_1)_* \sO_{B^{k-1}(L) \times C}.
	$$
	The Grothendieck spectral sequence then yields a map 
	$$
	R^i (\beta_{k-1} \circ \pr_1)_* \sO_{B^{k-1}(L) \times C} \longrightarrow R^{i-1} \beta_{k-1,*} R^1 \pr_{1,*}  \sO_{B^{k-1}(L) \times C}.
	$$
	But we have
	\begin{align*}
	R^{i-1} \beta_{k-1,*} R^1 \pr_{1,*}  \sO_{B^{k-1}(L) \times C} &= H^1(C, \sO_C) \otimes R^{i-1} \beta_{k-1,*} \sO_{B^{k-1}(L)} \\
    & =  H^1(C, \sO_C) \otimes\wedge^{i-1} H^1(C, \sO_C) \otimes \sO_{\Sigma_{k-i}},
	\end{align*}
	so we obtain a map
	\begin{equation}\label{eq:R^ibetaO_B}
		R^i \beta_{k,*} \sO_{B^k(L)} \longrightarrow \wedge^i H^1(C, \sO_C) \otimes \sO_{\Sigma_{k-i}}
	\end{equation}
	factoring through $H^1(C, \sO_C) \otimes \wedge^{i-1} H^1(C, \sO_C) \otimes \sO_{\Sigma_{k-i}}$.
	By the formal function theorem, this map can be identified with
	$$
	\varprojlim H^i (\sO_{B^k(L)}/I_{F_x}^{\ell}) \longrightarrow \wedge^i H^1(C, \sO_C) \otimes \varprojlim H^0(B^{k-i}(L), \sO_{B^{k-i}(L)}/I_{G_x}^{\ell}),
	$$
	where $G_x:=\beta_{k-i}^{-1}(x) = C_{k-i-m}$. 
	Recall that
	$$
	N_{F_x/B^k(L)}^{*} = \sO_{F_x}^{\oplus 2m+1} \oplus E_{k-m, L(-2\xi_x)}~~\text{ and }~~N_{G_x/B^{k-i}(L)}^{*} = \sO_{G_x}^{\oplus 2m+1} \oplus E_{k-i-m, L(-2\xi_x)}.
	$$
	Consider the commutative diagram\\[-20pt]
	
	\begin{scriptsize}
		$$
		\xymatrixcolsep{0.15in}
		\xymatrix{
			\cdots \ar[r] & H^i(I_{F_x}^{\ell}/I_{F_x}^{\ell+1}) \ar[r] \ar[d] & H^i(\sO_{B^k(L)}/I_{F_x}^{\ell+1}) \ar[r] \ar[d] & H^i(\sO_{B^k(L)}/I_{F_x}^{\ell}) \ar[d] \ar[r] & \cdots \\
			0 \ar[r] &\wedge^i H^1(\sO_C) \otimes H^0(I_{G_x}^{\ell}/I_{G_x}^{\ell+1}) \ar[r] & \wedge^i H^1(\sO_C) \otimes H^0(\sO_{B^{k-i}(L)}/I_{G_x}^{\ell+1}) \ar[r] & \wedge^i H^1(\sO_C) \otimes H^0(\sO_{B^{k-i}(L) }/I_{G_x}^{\ell}) \ar[r] & 0.
		}
		$$
	\end{scriptsize}
	
	\noindent Note that $I_{F_x}^{\ell}/I_{F_x}^{\ell+1} = S^{\ell} N_{F_x/B^k(L)}^{*}$ and $I_{G_x}^{\ell}/I_{G_x}^{\ell+1} = S^{\ell} N^*_{G_x/B^{k-i}(L)}$. By induction,
	\begin{align*}
		H^i(C_{k-m}, S^{\ell} E_{k-m, L(-2\xi_x)}) &= \wedge^i H^1(C, \sO_C) \otimes H^0(\Sigma_{k-m-i}, \sO_{\Sigma_{k-i}}(\ell)) \\
		& = \wedge^i H^1(C, \sO_C) \otimes H^0(C_{k-i-m}, S^{\ell} E_{k-i-m, L(-2\xi_x)}).
	\end{align*}
	Then
	$$
	H^i(I_{F_x}^{\ell}/I_{F_x}^{\ell+1}) = \wedge^i H^1(C, \sO_C) \otimes H^0(I_{G_x}^{\ell}/I_{G_x}^{\ell+1}).
	$$
	This implies by induction on $\ell$ that the map
	$$
	H^i(\sO_{B^k(L)}/I_{F_x}^{\ell+1}) \longrightarrow \wedge^i H^1(C, \sO_C) \otimes H^0(\sO_{B^{k-i}(L) }/I_{G_x}^{\ell+1})
	$$
	is an isomorphism. Hence the map (\ref{eq:R^ibetaO_B}) is an isomorphism.
\end{proof}

While $H^{k+1}(C_{k+1}, S^{\ell} E_{k+1, L}) =0$ for $\ell > 0$, it is possible that $H^{k+1}(C_{k+1}, \sO_{C_{k+1}}) \neq 0$. Indeed,  $H^i(C_{k+1}, \sO_{C_{k+1}})=\wedge^i H^1(C, \sO_C)$ for $0 \leq i \leq k+1$ by \cite[Lemma 3.7]{ENP}, and thus, $H^{k+1}(C_{k+1}, \sO_{C_{k+1}}) = \wedge^{k+1} H^1(C, \sO_C)$. 

\medskip

For a complete answer to \cite[Problem 6.2]{ENP}, we need to know the exact value of the Hilbert function
$$
H_{\Sigma_{k-i}}(\ell):= h^0(\Sigma_{k-i}, \sO_{\Sigma_{k-i}}(\ell))~~\text{ for $\ell > 0$}.
$$
As $H^j(\Sigma_{k-i}, \sO_{\Sigma_{k-i}}(\ell))=0$ for $j>0$ and $\ell > 0$ (see Theorem \ref{p:01}),  the Hilbert function $H_{\Sigma_{k-i}}(\ell)$ coincides with the Hilbert polynomial $\chi(\sO_{\Sigma_{k-i}}(\ell))$ for every $\ell > 0$.

\begin{lemma}\label{lem:longexactsequence}
For an integer $-k \leq \ell \leq -1$, there is an exact sequence
$$
\begin{array}{l}
	0 \longrightarrow \wedge^k H^1(C, \sO_C) \otimes H^1(\Sigma_0, \sO_{\Sigma_0}(\ell)) \longrightarrow \cdots \longrightarrow \wedge^2 H^1(C, \sO_C) \otimes H^{2k-3}(\Sigma_{k-2}, \sO_{\Sigma_{k-2}}(\ell))\\
	\longrightarrow  H^1(C, \sO_C) \otimes H^{2k-1}(\Sigma_{k-1}, \sO_{\Sigma_{k-1}}(\ell)) \longrightarrow H^{2k+1}(\Sigma_{k}, \sO_{\Sigma_{k}}(\ell)) \longrightarrow 0.
\end{array}
$$
\end{lemma}

\begin{proof}
Since $-k \leq \ell \leq -1$, it follows that $H^i(B^k(L), \sO_{B^k(L)}(\ell))=0$ for $0 \leq i \leq 2k+1$. By Corollary \ref{cor:Sigma=aCM}, for each $0 \leq m \leq k$, we have $H^i(\Sigma_{k-m}, \sO_{\Sigma_{k-m}}(\ell))=0$ for $0 \leq i \leq 2k-2m$. Then the Leray spectral sequence
$$
E_2^{p,q} = H^p(\Sigma_k, R^q \beta_{k,*} \sO_{B^k(L)}(\ell)) \Longrightarrow H^{p+q}(B^k(L), \sO_{B^k(L)}(\ell))=0
$$
yields that the complex from $E_2$-page, which is the sequence in the statement, is exact.
\end{proof}

We show a recursive formula for the Hilbert polynomial $\chi(\sO_{\Sigma_k}(m))$ of the secant variety $\Sigma_k$. 

\begin{theorem}\label{thm:hilbertpolynomial}
Assume that $d:=\deg L \geq 2g+2k+1$, and let $\Sigma_k:=\Sigma_k(L)$ be the $k$-th secant variety of the curve $C$ in the space $\nP H^0(C, L)$. For any $m \in \mathbb{Z}$, one has
	$$
	\chi(\sO_{\Sigma_k}(m))=\sum_{\ell=-k}^{k+1} (-1)^{k+1-\ell} a_{\ell}{2k+2 \choose k+\ell} {m+k \choose 2k+2}\frac{k+2-\ell}{m-\ell} ,
	$$
where
	$$
	a_{\ell}:=\begin{cases}
		\sum_{i=1}^k (-1)^i {g \choose i} \chi(\sO_{\Sigma_{k-i}}(\ell)) & \text{for $-k \leq \ell \leq -1$}\\
		1 - { g+k \choose k+1} & \text{for $\ell=0$}\\
		{ d-g+\ell \choose \ell} & \text{for $1 \leq \ell \leq k+1$}.
	\end{cases}
	$$
\end{theorem}

\begin{proof}
Recall that $H^i(\Sigma_k, \sO_{\Sigma_k}(\ell))=0$ for $1 \leq i \leq 2k$ and $\ell \in \nZ$ and $\chi(\sO_{\Sigma_k}(m))$ is a polynomial of degree $2k+1$ in $m$. To compute $\chi(\sO_{\Sigma_k}(m))$, it suffices to calculate $\chi(\sO_{\Sigma_k}(\ell))$ for $2k+2$ many different values of $\ell$. We treat the cases where $-k \leq \ell \leq k+1$. Recall from \cite[Theorem 1.2]{ENP} that $h^{2k+1}(\Sigma_k, \sO_{\Sigma_k}) = h^0(\Sigma_k, \omega_{\Sigma_k}) = {g+k \choose k+1}$.
Then we have
$$
\chi(\sO_{\Sigma_k})=h^0(\Sigma_k, \sO_{\Sigma_k})-h^{2k+1}(\Sigma_k, \sO_{\Sigma_k}) = 1 - { g+k \choose k+1} = a_0.
$$
For $ 1 \leq \ell \leq k+1$, Danila's theorem \cite{Danila} (see also \cite{ENP2}) gives
$$
\chi(\sO_{\Sigma_k}(\ell))=h^0(\Sigma_k, \sO_{\Sigma_k}(\ell)) = { d-g+\ell \choose \ell} = a_{\ell}.
$$
For $-k \leq \ell \leq -1$, Lemma \ref{lem:longexactsequence} gives
$$
\chi(\sO_{\Sigma_k}(\ell)) = -h^{2k+1}(\Sigma_k, \sO_{\Sigma_k}(\ell))=\sum_{i=1}^k (-1)^i {g \choose i} \chi(\sO_{\Sigma_{k-i}}(\ell)) = a_{\ell}.
$$
Then the assertion follows from the Lagrange interpolation formula.
\end{proof}

\begin{example}
We explicitly compute the Hilbert polynomial of $\Sigma_k$ for $k=0$ and $k=1$. 
First, consider the case of $k=0$. We have $a_0 = 1-g$ and $a_1=d-g+1$ Then
$$
\chi(\sO_C(m)) = -a_0 {2 \choose 0} {m \choose 2} \frac{2}{m} + a_1 {2 \choose 1}{ m \choose 2} \frac{1}{m-1} = dm-g+1.
$$
This coincides with Riemann--Roch formula. 
Next, consider the case of $k=1$. We have  
$$
a_{-1} = gd+g^2-g, ~a_0 = -\frac{g^2+g-2}{2},~a_1 = d-g+1,~a_2 = \frac{(d-g+2)(d-g+1)}{2}.
$$
Then
\begin{align*}
\chi(\sO_{\Sigma_1}(m)) = -(gd+g^2-g) \frac{m(m-1)(m-2)}{6} - (g^2+g-2)\frac{(m+1)(m-1)(m-2)}{4} \\ - (d-g+1)\frac{(m+1)m(m-2)}{2} + (d-g+2)(d-g+1) \frac{(m+1)m(m-1)}{12}.
\end{align*}
This was first computed by Sidman--Vermeire \cite[Theorem 1.1]{SV}.
\end{example}

\begin{remark}
From Theorem \ref{thm:hilbertpolynomial}, one can deduce that every coefficient of the Hilbert polynomial $\chi(\sO_{\Sigma_k}(m))$ is a polynomial in $m$ and $d$ once we fix $g$ and $k$. On the other hand, if $d=\deg L \geq 2g+2k+2$, then the defining ideal of $\Sigma_k$ in $\nP^r$ is generated by forms of degree $k+2$ (see \cite[Theorem 1.2]{ENP}). The number of minimal generators of the defining ideal of $\Sigma_k$ in $\nP^r$ is
\begin{align*}
h^0(\nP^r, I_{\Sigma_k/\nP^r}(k+2)) &= h^0(\nP^r, \sO_{\nP^r}(k+2)) - h^0(\Sigma_k, \sO_{\Sigma_k}(k+2))\\
&= {d-g+k+2 \choose k+2} - \chi(\sO_{\Sigma_k}(k+2)),
\end{align*}
which turns out to be a polynomial in $d$ once we fix $g$ and $k$. 
\end{remark}

One may also want to compute the cohomology groups of exterior powers of secant sheaves on $C_{k+1}$. For this purpose, we consider the addition map
$$
\sigma=\sigma_{k+1-\ell, \ell} \colon C_{k+1-\ell} \times C_{\ell} \longrightarrow C_{k+1}
$$
sending $(\xi, \eta)$ to $\xi+\eta$ for $1 \leq \ell \leq k+1$. Note that $\sigma$ is a finite and flat morhpism of degree ${k+1 \choose \ell}$. The pullback of the line bundle $T_{k+1}(L)$ by the map $\sigma$ is 
$$\sigma^*T_{k+1}(L)=T_{k+1-\ell}(L)\boxtimes T_{\ell}(L),$$
which can be verified by pulling back to the Cartesian product $C^{k+1}$ (see \cite[Remark 3.3(5)]{ENP}). We prove the following lemma, whose special case when $C=\nP^1$ was observed by  Raicu--Sam \cite[Theorem 3.1]{RS}.

\begin{lemma}\label{lem:sigma_*N}
For any line bundle $L$ on $C$, we have
$$
\wedge^{\ell} E_{k+1, L} = \sigma_* (\sO_{C_{k+1-\ell}} \boxtimes N_{\ell, L})~~\text{ for $1 \leq \ell \leq k+1$}.
$$
\end{lemma}

\begin{proof}
By \cite[Lemma 3.1]{Agostini:SecnatConj}, we have a map
$\sigma^* E_{k+1,L} \to \sO_{C_{k+1-\ell}} \boxtimes E_{\ell, L}$. Taking $\ell$-th exterior power, we get a map $\sigma^* \wedge^{\ell} E_{k+1,L} \to \sO_{C_{k+1-\ell}} \boxtimes N_{\ell, L}$. Pushing forward this map by $\sigma$ and composing with the splitting injective map $\wedge^{\ell} E_{k+1,L} \to \sigma_* \sigma^* \wedge^{\ell} E_{k+1,L}$, we obtain a map
$$
\varphi_L \colon \wedge^{\ell} E_{k+1, L} \longrightarrow \sigma_* (\sO_{C_{k+1-\ell}} \boxtimes N_{\ell, L}).
$$
If $\xi \in C_{k+1}$ consists of $k+1$ distinct points on $C$, then one can easily check that $\varphi_L \otimes \Bbbk(\xi)$ is an isomorphism because 
$$
\wedge^{\ell} E_{k+1, L}\otimes \Bbbk(\xi)=\wedge^\ell H^0(L|_{\xi})=\bigoplus_{\eta\subseteq \xi, \deg \eta = \ell }N_{\ell,L}\otimes \Bbbk(\eta).
$$ 
Thus $\varphi_L$ is generically injective, so in fact, $\varphi_L$ is injective. We need to confirm that $\varphi_L$ is surjective. For any $\xi \in C_{k+1}$, there is an affine open subset $U \subseteq C$ such that $\xi \subseteq U$ and $L|_U = \sO_U$. Then the $(k+1)$-th symmetric product $U_{k+1}$ of $U$ is an open subset of $C_{k+1}$, and $\varphi_L|_{U_{k+1}} = \varphi_{\sO_C}|_{U_{k+1}}$. Therefore, if we prove that $\varphi_L$ is surjective for one line bundle $L$, then the same holds for all line bundles $L$. We henceforth assume that $L$ is sufficiently positive. It suffices to show that the map 
\begin{equation}\label{eq:glmapwedgeEtoN}
H^0(C_{k+1}, \wedge^{\ell} E_{k+1,L} \otimes T_{k+1}(M)) \longrightarrow H^0(C_{k+1-\ell} \times C_{\ell}, T_{k+1-\ell}(M) \boxtimes N_{\ell, L \otimes M})
\end{equation}
is surjective for a sufficiently positive line bundle $M$ on $C$.

To this end, consider $C_{k+1-\ell} \times C_{\ell}$ embedded as a subvariety into $C_{k+1} \times C_{\ell}$ by sending $(\xi, \eta) \mapsto (\xi+\eta, \eta)$.
Notice that $H^0(L)=H^0(C_{k+1},E_{k+1,L})$. So for an integer $\ell\geq 0$, there is a natural evaluation map $\wedge^\ell H^0(L)\otimes \sO_{C_{k+1}}\rightarrow \wedge^\ell E_{k+1,L}$ of sheaves. Working on $C_{k+1}\times C_{\ell}$, we observe that 
$$(\wedge^\ell E_{k+1,L}\otimes T_{k+1}(M))\boxtimes \sO_{C_\ell}|_{C_{k+1-\ell}\times C_{\ell}}=\sigma^*(\wedge^\ell E_{k+1,L} \otimes T_{k+1}(M)), \text{ and}$$
$$T_{k+1,M} \boxtimes N_{\ell, L}|_{C_{k+1-\ell}\times C_{\ell}}=T_{k+1-\ell}(M) \boxtimes N_{\ell, L \otimes M}.$$
Thus we can form the following commutative diagram 
$$
\xymatrix{
\wedge^{\ell} H^0(L) \otimes (T_{k+1}(M) \boxtimes \sO_{C_{\ell}}) \ar[r] \ar[d]  & (\wedge^\ell E_{k+1,L} \otimes T_{k+1}(M)) \boxtimes \sO_{C_{\ell}} \ar[r] \ar[rd] & \sigma^* (\wedge^{\ell}E_{k+1,L} \otimes T_{k+1}(M)) \ar[d] \\
T_{k+1,M} \boxtimes N_{\ell, L} \ar[rr] & & T_{k+1-\ell}(M) \boxtimes N_{\ell, L \otimes M},
}
$$
where the left upper horizontal map and the left vertical map are the evaluation maps and the right upper horizontal map and the lower horizontal maps are the restriction maps to $C_{k+1-\ell} \times C_{\ell}$ . 
As the map (\ref{eq:glmapwedgeEtoN}) is the map on the global sections of the diagonal map, it is enough to check that the map on the global sections of the lower horizontal map is surjective. 
This can be deduced from 
$$
H^1(C_{k+1} \times C_{\ell}, I_{C_{k+1-\ell} \times C_{\ell}/C_{k+1} \times C_{\ell}} \otimes (T_{k+1}(M) \boxtimes N_{\ell, L}) = 0, 
$$
which holds true since both $M$ and $L$ are sufficiently positive.
\end{proof}

\begin{remark}
Assuming that $\varphi_L$ is an isomorphism for a sufficiently positive line bundle $L$, one can alternatively prove that $\varphi_L$ is an isomorphism for every line bundle $L$ as follows.
For a point $p \in C$, consider the commutative diagram
$$
\xymatrix{
0 \ar[r] & \wedge^{\ell} E_{k+1,L} \ar[d]^-{\varphi_L}\ar[r] & \wedge^{\ell} E_{k+1,L(p)} \ar[d]^-{\varphi_{L(p)}}\ar[r] & \wedge^{\ell-1} E_{k, L} \ar[d] \ar[r] & 0\\
0 \ar[r] & \sigma_*(\sO_{C_{k+1-\ell}} \boxtimes N_{\ell, L}) \ar[r] & \sigma_*(\sO_{C_{k+1-\ell}} \boxtimes N_{\ell, L(p)}) \ar[r] & \sigma_*(\sO_{C_{k+1-\ell}} \boxtimes N_{\ell-1, L}) \ar[r] & 0,
}
$$
where the two rows are short exact sequences and the two right-most vector bundles are supported on $X_p:=\{ \xi \in C_{k+1} \mid p \in \xi\} \cong C_k$. By induction on $k$, we may assume that the right-most vertical map is an isomorphism. Then $\varphi_L$ is an isomorphism if and only if $\varphi_{L(p)}$ is an isomorphism. Thus we can increase $\deg L$ as much as we want, and then, the proof is complete by the initial assumption. 
\end{remark}

\begin{theorem}\label{thm:cohwedgesecsheaf}
For line bundles $L$ and $M$ on $C$, we have
\begin{align*}
&H^i(C_{k+1}, \wedge^{\ell} E_{k+1,L} \otimes T_{k+1}(M)) \\
&= \bigoplus_{p+q=i}  \big( S^{k+1-\ell-p}H^0(C, M) \otimes S^q H^1(C, L \otimes M) \otimes \wedge^p H^1(C,M) \otimes \wedge^{\ell-q} H^0(C, L \otimes M) \big)
\end{align*}
for $0 \leq i \leq k+1$ and $1 \leq \ell \leq k+1$. In particular,
$$
H^i(C_{k+1}, \wedge^{\ell} E_{k+1,L}) = \bigoplus_{p+q=i}  \big(  S^q H^1(C, L ) \otimes \wedge^p H^1(C,\sO_C) \otimes \wedge^{\ell-q} H^0(C, L) \big)
$$
for $0 \leq i \leq k+1$ and $1 \leq \ell \leq k+1$.
\end{theorem}

\begin{proof}
Note that $\sigma^* T_{k+1}(M) = T_{k+1-\ell}(M) \boxtimes T_{\ell}(M)$. 
By Lemma \ref{lem:sigma_*N}, we have
$$
\wedge^{\ell} E_{k+1,L} \otimes T_{k+1}(M) = \sigma_* (T_{k+1-\ell}(M) \boxtimes N_{\ell, L \otimes M}). 
$$
The assertion follows from K\"{u}nneth formula and Lemma \ref{lm:03}.
\end{proof}

We remark that we do not impose any conditions on the line bundles $L$ and $M$ for the above theorem.

\appendix
\section{Cohomological approach to Cohen--Macaulayness}\label{sec:appendix}



\noindent In this appendix, we give a complete proof of \cite[Theorem 5.8]{ENP}, which asserts that the $k$-th secant variety $\Sigma_k \subseteq \nP^r = \nP H^0(L)$ is arithmetically Cohen--Macaulay whenever $\deg L \geq 2g+2k+1$. We closely follow the cohomological approach of \cite{ENP}, but we do not quote \cite[Theorem 10.42]{K}. Of course, we already gave an alternative way to show this result in Section \ref{sec:tangentcones} (see Corollary \ref{cor:Sigma=aCM}). But it would be interesting to see both methods work for the result.

\begin{theorem}\label{p:41} 
Suppose that $\deg L\geq 2g+2k+1$ and consider the exact triangle
		\begin{equation*}
		\omega^\bullet_{\Sigma_{k-1}}\longrightarrow \omega^\bullet_{\Sigma_k}\longrightarrow R\beta_{k,*}\omega_{B^k(L)}(Z_{k-1})[2k+1]\stackrel{+1}{\longrightarrow}
	\end{equation*} 
in the derived category of the $k$-th secant variety $\Sigma_k=\Sigma_k(C, L)$.
	\begin{enumerate}
		\item The induced map on the $(-2k-1)$-th cohomology sheaves $$\omega_{\Sigma_k}\longrightarrow \beta_{k,*}\omega_{B^k(L)}(Z_{k-1})$$ is an isomorphism.
		\item The connection map on the  cohomology sheaves $$R^1\beta_{k,*}\omega_{B^k(L)}(Z_{k-1})\longrightarrow \omega_{\Sigma_{k-1}}$$ is an isomorphism.
		\item $R^i\beta_{k,*}\omega_{B^k(L)}(Z_{k-1})=0$, for $i\geq 2$.
	\end{enumerate}
	In particular, $H^{-i}(\omega^\bullet_{\Sigma_k})=0$ for $i\neq 2k+1$. Therefore $\Sigma_k$ is Cohen--Macaulay, and as a consequence, $\Sigma_k \subseteq \nP^r = \nP H^0(L)$ is arithmetically Cohen--Macaulay.
\end{theorem}
\begin{proof} 
	We prove the theorem by induction on $k$. If $k=0$, then the zero-th secant variety is exactly the curve $C$ itself and $B^0(L)=C$ so the map $\beta_0$ is just the embedding of $C$ into $\nP^r$. The results (1), (2) and (3) follow immediately. The curve $C$ is arithmetically Cohen--Macaulay since $\deg L\geq 2g+1$. 
	
	Next, we assume $k\geq 1$ and assume that for any secant variety $\Sigma_m(\sL)$ with $m\leq k-1$ associated to a line bundle $\sL$ of degree $\geq 2g+2m+1$, the proposition holds. Now we consider the secant variety $\Sigma_k$ embedded by the line bundle $L$ of $\deg L\geq 2g+2k+1$. Since $\dim\Sigma_{k-1}=2k-1$ and $\Sigma_{k-1}$ is Cohen--Macaulay by the induction hypothesis, we have  $h^i(\omega^\bullet_{\Sigma_{k-1}})=0$ if $i\neq -2k+1$ and $h^{-2k+1}(\omega^\bullet_{\Sigma_{k-1}})=\omega_{\Sigma_{k-1}}$.  We take $(-2k-1)-$th cohomology of the triangle  (\ref{eq:tri}) to get
	$$0\longrightarrow H^{-2k-1}(\omega^\bullet_{\Sigma_k})\longrightarrow H^{-2k-1}(R\beta_{k,*}\omega_{B^k(L)}(Z_{k-1}))\longrightarrow 0,$$
	which gives the isomorphism in (1) that 
	$$\omega_{\Sigma_{k}}\cong \beta_*\omega_{B^k(L)}(Z_{k-1}).$$
	(Here recall $\omega_{\Sigma_{k}}=H^{-2k-1}(\omega^\bullet_{\Sigma_k})$ by definition.)
	We take $(-2k)-$th cohomology of the triangle (\ref{eq:tri}) to get	
	$$0\longrightarrow H^{-2k}(\omega^\bullet_{\Sigma_k})\longrightarrow R^1\beta_{k,*}\omega_{B^k(L)}(Z_{k-1})\longrightarrow \omega_{\Sigma_{k-1}}\longrightarrow H^{-2k+1}(\omega^\bullet_{\Sigma_k})\longrightarrow \cdots.$$ 
	So we obtain a natural connection morphism 
	$$\theta \colon R^1\beta_{k,*}\omega_{B^k(L)}(Z_{k-1})\longrightarrow \omega_{\Sigma_{k-1}}$$
	claimed in (2).
	
	
	In the sequel, we shall use the formal function theorem to establish the isomorphism of $\theta$ in (2) and the statement (3). To this end, we work at a closed point $x\in \Sigma_{k}$. As $\beta_k$ is an isomorphism over the open set $\Sigma_{k} \setminus \Sigma_{k-1}$, we may only consider the point $x\in \Sigma_{k-1}$. Then we must have $x \in \Sigma_m \setminus \Sigma_{m-1}$ for some  $0 \leq m \leq k-1$. Furthermore there is an effective divisor $\xi$ on $C$ of degree $m+1$, i.e., $\xi \in C_{m+1}$, uniquely determined by $x$ such that  $x$ lies in the linear span by $\xi$ in $\nP^r$.  Note that $\deg L(-2\xi) \geq 2g+2(k-m-1)+1$.
	
	Let $F:=\beta_k^{-1}(x)\cong C_{k-m}$ be the fiber of $\beta_k$ over the point $x$ and let $I_F$ be the defining ideal of $F$ in $B^k(L)$. For $n\geq 1$, write $F_n$ for the thickening fibers of $F$, i.e., $F_n$ is a subscheme of $B^k(L)$ defined by the ideal sheaf $I^n_{F}$. We write the sheaves 
	$$\omega_{B^k}(Z_{k-1})_n:=\omega_{B^k}(Z_{k-1})\otimes \sO_{F_n}.$$
	Note that $$\omega_{B^k}(Z_{k-1})_1=\omega_{B^k}(Z_{k-1})\otimes \sO_{F}=T_{k-m}(\omega_C).$$
	By \cite[Proposition 3.13]{ENP}, the conormal sheaf of $F$ inside $B^k(L)$ is 
	\begin{equation}\label{eq:12}
		N_{F/B^k(L)}^{*} =I_F/I_F^2= \sO_F^{\oplus 2m+1} \oplus E_{k-m, L(-2\xi)}.
	\end{equation}
	As $I_F^{\ell}/I_F^{\ell+1} = S^\ell N_{F/B^k(L)}^{*}$, we have a short exact sequence 
	\begin{equation}\label{eq:11}
		\xymatrix{
			0\ar[r]& S^nN^*_F\otimes \omega_{B^k}(Z_{k-1}) \ar[r]  & \omega_{B^k}(Z_{k-1})_{n+1} \ar[r] & \omega_{B^k}(Z_{k-1})_n \ar[r]& 0
		}
	\end{equation}

	Let us first show the statement (3). For $i\geq 2$, the formal function theorem says that the completion of the sheaf $R^i \beta_{k,*} \omega_{B^k(L)}(Z_{k-1})$ at $x$ equals
	$$\varprojlim_n H^i(F_n,  \omega_{B^k(L)}(Z_{k-1})_n).
	$$
	Thus in order to show the vanishing $R^i \beta_{k,*} \omega_{B^k(L)}(Z_{k-1}) = 0$ it is enough to show 
	\begin{equation}\label{eq:13}
		H^i(F_n, \omega_{B^k(L)}(Z_{k-1})_n)=0, \text{ for } i\geq 2.
	\end{equation}
	For this, we do the induction on $n$. So when $n=1$, we have	 
	$$H^i(F_1, \omega_{B^k(L)}(Z_{k-1})_1)=H^i(C_{k-m}, T_{k-m}(\omega_C))=0, \text{ for }i\geq 2,$$
	by Lemma \ref{lm:03}. To proceed for arbitrary $n$, we use the short exact sequence  (\ref{eq:11}) and it suffices to show 
	\begin{equation}\label{eq:14}
		H^i(S^nN^*_F\otimes \omega_{B^k}(Z_{k-1}))=0, \text{  for } i\geq 2,\ n\geq 1.
	\end{equation}
	But because of (\ref{eq:12}), this can be further reduced to show 
	$$
	H^i(C_{k-m}, S^{\ell} E_{k-m, L(-2\xi)} \otimes T_{k-m}(\omega_C))=0~~\text{ for $i \geq 2, \ \ell\geq 1$}.
	$$
	
	To this end, we consider secant varieties 
	\begin{equation}\label{eq:06}
		\Sigma_{k-m-1}:=\Sigma_{k-m-1}(L(-2\xi)) \text{ and }\Sigma_{k-m-2}:=\Sigma_{k-m-2}(L(-2\xi))
	\end{equation}
	in the space $\nP(H^0(L(-2\xi)))$ associated to the curve $C$ embedded by the line bundle $L(-2\xi)$. As $\deg L(-2\xi) \geq 2g+2(k-m-1)+1$,  the inductive hypothesis applies and the proposition holds for both $\Sigma_{k-m-1}$ and $\Sigma_{k-m-2}$. In particular they both are arithmetical Cohen--Macaulay and therefore we  have the vanishing  for all $\ell\in \nZ$, $$H^i(\Sigma_{k-m-1}, \sO_{\Sigma_{k-m-1}}(\ell))=0 \text{ for } 1\leq i\leq 2(k-m-1), \text{ and }$$ 
	$$H^i(\Sigma_{k-m-2}, \sO_{\Sigma_{k-m-2}}(\ell))=0 \text{ for } 1\leq i\leq 2(k-m-2).$$
	Chasing through the short exact sequence 
	$$0\longrightarrow I_{\Sigma_{k-m-2}|\Sigma_{k-m-1}}\longrightarrow \sO_{\Sigma_{k-m-1}}\longrightarrow \sO_{\Sigma_{k-m-2}}\longrightarrow 0$$ 
	we obtain that for all $\ell\in \nZ$
	\begin{equation}\label{eq:05}
		H^{2(k-m-1)+1-i}(\Sigma_{k-m-1}, I_{\Sigma_{k-m-2}|\Sigma_{k-m-1}}(-\ell))^{*} = 0, \text{ for } i \geq 2, \text{ and }
	\end{equation}
	\begin{equation}\label{eq:07}
		H^{2(k-m-1)}(\Sigma_{k-m-1}, I_{\Sigma_{k-m-2}|\Sigma_{k-m-1}}(-\ell))^{*}=H^{2(k-m-1)-1}(\Sigma_{k-m-2}, \sO_{\Sigma_{k-m-2}}(-\ell))^{*}.
	\end{equation}
	Now with Lemma \ref{lm:01}(1) and the vanishing (\ref{eq:05}), we conclude 
	$$
	H^i(C_{k-m}, S^{\ell} E_{k-m, L(-2\xi)} \otimes T_{k-m}(\omega_C)) = H^{2(k-m-1)+1-i}(\Sigma_{k-m-1}, I_{\Sigma_{k-m-2}|\Sigma_{k-m-1}}(-\ell))^{*} = 0
	$$
	for $i \geq 2$ and $\ell\geq 1$, which completes the proof of (3).

	Next we prove the statement (2). Consider the commutative diagram 	
\begin{equation}\label{eq:01}
	\begin{split}
		\xymatrixcolsep{0.7in}
	\xymatrix{
		B^{k-1}(L)\times C \ar[r]^-{\alpha_{k,k-1}} \ar[d]_-{\pr_1} & Z_{k-1} \ar[d]^-{\beta_k} \\
		B^{k-1}(L) \ar[r]_-{\beta_{k-1}} & \Sigma_{k-1}.
	}
\end{split}
\end{equation}
Let $F':=\beta_{k-1}^{-1}(x)$ and let $F'_n$ be the thickening fibers defined by the $n$-th power of the ideal sheaf of $F'$ in $B^{k-1}(L)$. We write the sheaf 
	$$\omega_{B^{k-1}}(Z_{k-2})_n:=\omega_{B^{k-1}}(Z_{k-2})\otimes \sO_{F'_n}$$
	for the restriction of $\omega_{B^{k-1}}(Z_{k-2})$ on $F'_n$.	Let $N^*_{F'}=I_{F'}/I_{F'}$ be the conormal sheaf of $F'$ inside $B^{k-1}(L)$. Let $x_n$ be the fat point defined by the $n$-th power of the defining ideal of the point $x$ in $\Sigma_{k}$. Then induced from (\ref{eq:01}), we have a commutative diagram 
	$$
	\xymatrixcolsep{0.7in}
	\xymatrix{
		F'_n\times C \ar[r]^-{\alpha_{n}} \ar[d]_-{\pr_1} & F_n \ar[d] \\
		F'_n \ar[r] & x_n.
	}
	$$
	in which $\alpha_n$ is finite by Lemma \ref{lm:01}(2). As $\alpha^*_{k,k-1}\omega_{B^k}(Z_{k-1})=\omega_{B^{k-1}}(Z_{k-2})\boxtimes \omega_C,$
	we see that 
	$$\alpha^*\omega_{B^k}(Z_{k-1})_n=\omega_{B^{k-1}}(Z_{k-2})_n\boxtimes \omega_C,$$
	and thus there is a natural map 
	$$H^i(\omega_{B^k}(Z_{k-1})_n)\longrightarrow H^i(\omega_{B^{k-1}}(Z_{k-2})_n\boxtimes \omega_C).$$
	Using K\"{u}nneth formula for $H^i(\omega_{B^{k-1}}(Z_{k-2})_n\boxtimes \omega_C)$, we get a natural map 
	$$H^i(\omega_{B^k}(Z_{k-1})_n)\longrightarrow  H^{i-1}(\omega_{B^{k-1}}(Z_{k-2})_n)\otimes H^1(\omega_C)=H^{i-1}(\omega_{B^{k-1}}(Z_{k-2})_n).$$
	Here we only need these natural maps when $i=1$. To put them in the inverse systems, we pull the short exact sequence 
	$$0\longrightarrow S^nN^*_{F'}\otimes \omega_{B^{k-1}}(Z_{k-2})\longrightarrow \omega_{B^{k-1}}(Z_{k-2})_{n+1}\longrightarrow\omega_{B^{k-1}}(Z_{k-2})_n\longrightarrow 0 $$
	to $B^{k-1}\times C$ by $pr_1$ and then tensor with $pr^*_2\omega_C$ to get
	$$0\longrightarrow S^nN^*_{F'}\otimes \omega_{B^{k-1}}(Z_{k-2})\boxtimes \omega_C\longrightarrow \omega_{B^{k-1}}(Z_{k-2})_{n+1}\boxtimes \omega_C\longrightarrow \omega_{B^{k-1}}(Z_{k-2})_n\boxtimes \omega_C\longrightarrow 0.$$
	Since the left-hand-side two sheaves are $\alpha^*\omega_{B^k}(Z_{k-1})_{n+1}$ and $\alpha^*\omega_{B^k}(Z_{k-1})_n$, we push the above short exact sequence down to $B^k(L)$ by $\alpha$  to get a commutative diagram 
	$$\tiny\xymatrix{
		0\ar[r]& S^nN^*_F\otimes \omega_{B^k}(Z_{k-1}) \ar[r] \ar[d] & \omega_{B^k}(Z_{k-1})_{n+1} \ar[r]\ar[d] & \omega_{B^k}(Z_{k-1})_n \ar[r]\ar[d]& 0\\
		0 \ar[r] &	\alpha_*(S^nN^*_{F'}\otimes \omega_{B^{k-1}}(Z_{k-2})\boxtimes \omega_C) \ar[r] &  \alpha_*(\omega_{B^{k-1}}(Z_{k-2})_{n+1}\boxtimes \omega_C) \ar[r] & \alpha_*(\omega_{B^{k-1}}(Z_{k-2})_n\boxtimes \omega_C)\ar[r] & 0
	}$$
	Taking cohomology and using K\"{u}nneth formula, we then obtain 
	\begin{equation}\label{eq:02}
		\begin{split}\tiny
		\xymatrix{
			\ar[r]& H^1(S^nN^*_F\otimes \omega_{B^k}(Z_{k-1}) )\ar[r] \ar[d]_-{a_n} &H^1( \omega_{B^k}(Z_{k-1})_{n+1}) \ar[r]\ar[d]_{c_{n+1}} & H^1(\omega_{B^k}(Z_{k-1})_n) \ar[r]\ar[d]_{c_n}& 0\\
			0 \ar[r] &	H^0(S^nN^*_{F'}\otimes \omega_{B^{k-1}}(Z_{k-2})) \ar[r] &  H^0(\omega_{B^{k-1}}(Z_{k-2})_{n+1}) \ar[r] & H^0(\omega_{B^{k-1}}(Z_{k-2})_n)\ar[r] & 
		}
		\end{split}		
	\end{equation}
	in which recall that  we knew in (\ref{eq:14}) that $H^i(S^nN^*\otimes \omega_{B^k}(Z_{k-1}) )=0$ for $i\geq 2$. 
	
	In order to show the natural map $\theta$ is isomorphic, we use the formal function theorem to show the induced map on the completion at the point $x$
	$$\hat{\theta}_x:
	\varprojlim H^1(F_n, \omega_{B^k(L)}(Z_{k-1})_n) \longrightarrow \varprojlim H^0(F'_n, \omega_{B^{k-1}(L)}(Z_{k-2})_n) 
	$$
	is an isomorphism. It is then enough to show the map
	$$c_n:H^1(\omega_{B^k}(Z_{k-1})_n)\longrightarrow H^0(\omega_{B^{k-1}}(Z_{k-2})_n)$$
	is an isomorphism for $n\geq 1$. Using  the diagram (\ref{eq:02}) above and proceeding by induction on $n$, we reduce to show the map  $c_1$ and the map $a_n$ for $n\geq 1$ are isomorphism. Notice first that 
	$$\omega_{B^k(L)}(Z_{k-1})_1=T_{k-m}(\omega_C)\text{ and }\omega_{B^{k-1}(L)}(Z_{k-2})_1=T_{k-1-m}(\omega_C).$$
	So the map $c_1$ turns out to be 
	$$c_1:H^1(C_{k-m,}T_{k-m}(\omega_C))\longrightarrow H^0(C_{k-1-m}, T_{k-1-m}(\omega_C))$$
	which is an isomorphism since both $H^1$ and $H^0$ groups involved are the same as  $S^{k-m-1}H^0(\omega_C)$ by Lemma \ref{lm:03}. Finally, to show the map $a_n$ with $n\geq 1$ is isomorphism, we notice that 
	$$N^*_F=\sO_F^{\oplus 2m+1}\oplus E_{k-m,L(-2\xi)}, \text{ and } N^*_{F'}=\sO_{F'}^{\oplus 2m+1}\oplus E_{k-1-m,L(2\xi)}.$$
	So it is enough to show the the isomorphism 
	\begin{equation}\label{eq:03}
			H^1(C_{k-m}, S^{\ell} E_{k-m, L(-2\xi)} \otimes T_{k-m}(\omega_C)) \cong H^0(C_{k-m-1}, S^{\ell} E_{k-m-1, L(-2\xi)} \otimes T_{k-m-1}( \omega_C))
	\end{equation}
	for all $\ell \geq 0$. To this end, once again we consider secant varieties (\ref{eq:06}) in the space $\nP(H^0(L(-2\xi)))$. We show that the spaces in (\ref{eq:03}) are the same as $ H^0(\Sigma_{k-m-2}, \omega_{\Sigma_{k-m-2}}(\ell))$. Indeed, on one hand, recall $B^{k-m-2}:=B^{k-m-2}(L(-2\xi))$ is $\nP(E_{k-m-1,L(-2\xi)})$ over $C_{k-m-1}$ by construction so we have 
	\begin{equation}\label{eq:04}\small
		H^0(C_{k-m-1}, S^{\ell} E_{k-m-1, L(-2\xi)} \otimes T_{k-m-1}(\omega_C))= H^0(B^{k-m-2}, \omega_{B^{k-m-2}}(Z_{k-m-3}) \otimes H_{k-m-2}^{\ell}),
	\end{equation}
	where $H$ is the tautological bundle of $B^{k-m-2}$. We also have a birational map
	$$\beta_{k-m-2}:B^{k-m-2}\longrightarrow \Sigma_{k-m-2}$$
	which satisfies $\beta_{k-m-2,*}\omega_{B^{k-m-2}}(Z_{k-m-3})=\omega_{\Sigma_{k-m-2}}$ by the inductive hypothesis. So we have 
	$$H^0(B^{k-m-2}, \omega_{B^{k-m-2}}(Z_{k-m-3}) \otimes H_{k-m-2}^{\ell})= H^0(\Sigma_{k-m-2}, \omega_{\Sigma_{k-m-2}}(\ell))$$
	and together with (\ref{eq:04}) we obtain 
	$$H^0(C_{k-m-1}, S^{\ell} E_{k-m-1, L(-2\xi)} \otimes T_{k-m-1}(\omega_C))=H^0(\Sigma_{k-m-2}, \omega_{\Sigma_{k-m-2}}(\ell)).$$
	On the other hand, using Serre's duality we have 
	$$H^{2(k-m-1)-1}(\Sigma_{k-m-2}, \sO_{\Sigma_{k-m-2}}(-\ell))^{*}=H^0(\Sigma_{k-m-2}, \omega_{\Sigma_{k-m-2}}(\ell)).$$
	Then it follows from Lemma \ref{lm:01}(1) and (\ref{eq:07}) that 
	$$H^1(C_{k-m}, S^{\ell} E_{k-m, L(-2\xi)} \otimes T_{k-m}(\omega_C))=H^0(\Sigma_{k-m-2}, \omega_{\Sigma_{k-m-2}}(\ell)),$$
	which finishes the proof of (\ref{eq:03}) and completes the proof of (2).
	
	Using (1), (2), (3) and chasing trough the exact triangle (\ref{eq:tri}), we obtain immediately that  $H^{-i}(\omega^\bullet_{\Sigma_k})=0$ for $i\neq 2k+1$. Therefore $\Sigma_k$ is Cohen--Macaulay. Finally, by Lemma \ref{lem:CM=>aCM}, we conclude that $\Sigma_k$ is arithmetially Cohen--Macaulay. 
\end{proof}

\begin{remark}
By the same way as in Section \ref{sec:cohomology}, we can compute the cohomology group $H^i(C_{k+1}, S^{\ell} E_{k+1, L} \otimes T_{k+1}(\omega_C))$ for $\ell > 0$. First, notice that
$$
H^i(C_{k+1}, S^{\ell} E_{k+1, L} \otimes T_{k+1}(\omega_C)) = H^i(B^k(L), \sO_{B^k(L)}(\ell) \otimes \omega_{B^k(L)}(Z_{k-1})).
$$
Theorem \ref{p:41} says that 
$$
R^i \beta_{k,*} \omega_{B^k(L)}(Z_{k-1}) = \begin{cases} \omega_{\Sigma_k} & \text{for $i=0$} \\ \omega_{\Sigma_{k-1}} & \text{for $i=1$} \\ 0 & \text{for $i \geq 2$}. \end{cases}
$$
As $\Sigma_k$ is arithmetically Cohen--Macaulay, we have $H^j(\Sigma_k, \omega_{\Sigma_k}(\ell))=0$ for $j>0$ and $\ell > 0$. Considering the Leray spectral sequence, we obtain
$$
H^i(B^k(L), \sO_{B^k(L)}(\ell) \otimes \omega_{B^k(L)}(Z_{k-1})) = \begin{cases} 
	H^0(\Sigma_k, \omega_{\Sigma_k}(\ell)) & \text{for $i=0$}\\
	H^0(\Sigma_{k-1}, \omega_{\Sigma_{k-1}}(\ell)) & \text{for $i=1$}\\
	0 & \text{for $i \geq 2$}.
\end{cases}
$$
By Serre duality, $H^0(\Sigma_k, \omega_{\Sigma_k}(\ell)) = H^{2k+1}(\Sigma_k, \sO_{\Sigma_k}(-\ell))^{*}$, so $h^0(\Sigma_k, \omega_{\Sigma_k}(\ell)) = -\chi(\sO_{\Sigma_k}(-\ell))$. Similarly, we have $h^0(\Sigma_{k-1}, \omega_{\Sigma_{k-1}}(\ell)) = -\chi(\sO_{\Sigma_{k-1}}(-\ell))$. 
\end{remark}

\bibliographystyle{alpha}
\bibliographystyle{alpha}

\end{document}